\documentclass[a4paper]{amsart}

\usepackage{amssymb,amscd}
\usepackage{mathrsfs} 
\usepackage[all]{xy}

\usepackage[
backref,
pdfauthor={Han Wu},  % add other authors
pdftitle={Non-invariance of the Brauer-Manin obstruction for surfaces},
]{hyperref}
\usepackage[alphabetic,backrefs,lite]{amsrefs}  % for bibliography, can choose nobysame

\theoremstyle{plain}
\theoremstyle{definition}

\newtheorem{theorem}{Theorem}[subsection]
\newtheorem{thm}{Theorem}[subsubsection]
\newtheorem{lemma}[theorem]{Lemma}

\newtheorem{prop}[theorem]{Proposition}

\theoremstyle{definition}
\newtheorem{definition}[theorem]{Definition}

\newtheorem{qu}[theorem]{Question}

\newtheorem{eg}[theorem]{Example}

\newtheorem{conjecture}[theorem]{Conjecture}

\theoremstyle{remark}
\newtheorem{remark}[theorem]{Remark}

\renewcommand{\AA}{\mathbb{A}}

\newcommand{\QQ}{\mathbb{Q}}
\newcommand{\RR}{\mathbb{R}}
\newcommand{\ZZ}{\mathbb{Z}}

\newcommand{\GG}{\mathbb{G}}
\newcommand{\PP}{\mathbb{P}}

\newcommand{\Lcal}{{\mathcal L}}
\newcommand{\Ocal}{{\mathcal O}}

\newcommand{\Ifr}{{\mathfrak I}}
\newcommand{\mfr}{{\mathfrak m}}

\newcommand{\Lscr}{{\mathscr L}}

\DeclareMathOperator{\Gal}{Gal}

\DeclareMathOperator{\inv}{inv}

\DeclareMathOperator{\Spec}{Spec}

\usepackage[OT2,T1]{fontenc}
\DeclareSymbolFont{cyrletters}{OT2}{wncyr}{m}{n}
\DeclareMathSymbol{\Sha}{\mathalpha}{cyrletters}{"58}
\newcommand{\defi}[1]{\textsf{#1}} % for defined terms

\newcommand{\BM}{Brauer-Manin }

\newcommand{\Br}{\textup{Br}}
\newcommand{\et}{\textup{\'et}}

\makeatletter
\g@addto@macro\bfseries{\boldmath}  % This makes math in section titles bold.
\makeatother

\begin{document}
	
	\begin{title}
		{Non-invariance of the Brauer-Manin obstruction for surfaces}  %\'etale
	\end{title}
	\author{Han Wu}
	\address{University of Science and Technology of China,
		School of Mathematical Sciences,
		No.96, JinZhai Road, Baohe District, Hefei,
		Anhui, 230026. P.R.China.}
	\email{wuhan90@mail.ustc.edu.cn}
	\date{}
	%\thanks{The author was partially supported by USTC}
	% 11G05, , 14H25, 14H52, 14K15, 14J30

	%\thanks{The authors were partially supported by University of Science and Technology of China}
	%\thanks{\textit{MSC 2010} : 11G35 14G05  14G25 14J20}

	% % % ----------------------------------------------------------------------

	% % % ----------------------------------------------------------------------

	\begin{abstract} 
		In this paper, we study the properties of weak approximation 
		with \BM obstruction %, \'etale \BM obstruction 
		and the Hasse principle with \BM obstruction for surfaces with respect to field extensions of number fields. We assume a conjecture of M. Stoll.
		For any nontrivial extension of number fields $L/K,$ we construct two kinds of smooth, projective, and geometrically connected surfaces defined over $K.$ For the surface of the first kind, it has a $K$-rational point, and satisfies weak approximation with \BM obstruction %, \'etale \BM obstruction 
		off $\infty_K,$ while its base change by $L$ does not so off $\infty_L.$ 
		For the surface of the second kind, it is a counterexample to the Hasse principle explained by the \BM obstruction, while the failure of the Hasse principle of its base change by $L$ cannot be so. We illustrate these constructions with explicit unconditional examples.
	\end{abstract} 
	
	\maketitle

	\section{Introduction}
	
	\subsection{Background}
	For a proper scheme $X$ over a number field $K,$ if its $K$-rational points set $X(K)\neq\emptyset,$ then its adelic points set $X(\AA_K)\neq\emptyset.$  The converse, as has been known, does not always hold. We say that $X$ is a \defi{counterexample to the Hasse principle} if the set $X(\AA_K)\neq\emptyset$ whereas the set $X(K)=\emptyset.$ Let $S\subset \Omega_K$ be a finite subset. By the diagonal embedding, we always view $X(K)$ as a subset of $X(\AA_K)$ (respectively of $X(\AA_K^S)$). 
	We say that $X$
	satisfies \defi{weak approximation} (respectively \defi{weak approximation off $S$}) if $X(K)$ is dense in $X(\AA_K)$ (respectively in $X(\AA_K^S)$), cf. \cite[Chapter 5.1]{Sk01}. Manin \cite{Ma71} used the Brauer group of $X$ to define a closed subset $X(\AA_K)^{\Br}\subset X(\AA_K),$ and showed that this closed subset can explain some failures of the Hasse principle and nondensity of $X(K)$ in $X(\AA_K^S).$  The global reciprocity law gives an inclusion: $X(K)\subset X(\AA_K)^\Br.$ 
	We say that the failure of the Hasse principle of $X$ is \defi{explained by the Brauer-Manin obstruction} if the set $X(\AA_K)\neq\emptyset$ and the set $X(\AA_K)^{\Br}=\emptyset.$ 
	We say that $X$ satisfies \defi{weak approximation with Brauer-Manin obstruction} (respectively \defi{with Brauer-Manin obstruction off $S$}) if $X(K)$ is dense in $X(\AA_K)^\Br$ (respectively in $pr^S(X(\AA_K)^\Br)$). 
	For a smooth, projective, and geometrically connected curve $C$ defined over a number field $K,$ assume that the Tate-Shafarevich group and the rational points set of its Jacobian are both finite. By the dual sequence of Cassels-Tate, Skorobogatov \cite[Chapter 6.2]{Sk01} and Scharaschkin \cite{Sc99} independently observed that $C(K)=pr^{\infty_K}(C(\AA_K)^\Br).$ In particular, if this curve $C$ is a counterexample to the Hasse principle, 
	then this failure can be explained by the Brauer-Manin obstruction. Stoll \cite{St07} generalized this observation, and made a conjecture that for any smooth, projective, and geometrically connected curve, it satisfies weak approximation with Brauer-Manin obstruction off $\infty_K\colon$ see Conjecture \ref{conjecture Stoll} for more details.
	
	\subsection{Questions}\label{Questions}
	Let $L/K$ be a nontrivial extension of number fields. Let $S\subset\Omega_K$ be a finite subset, and  let $S_L\subset \Omega_L$ be the subset of all places above $S.$ Given a smooth, projective, and geometrically connected variety $X$ over $K,$ let $X_L=X\times_{\Spec K} {\Spec L}$ be its base change by $L.$ In this paper, we consider the following questions.
	\begin{qu}\label{question on WA1}
		If the variety $X$ has a $K$-rational point, and satisfies weak approximation with Brauer-Manin obstruction off $S,$ must $X_L$ also satisfy weak approximation with Brauer-Manin obstruction off $S_L?$	
	\end{qu}
	%\begin{qu}\label{question on WA2}
	%	If a smooth, projective, and geometrically connected  variety $X$ with a $K$-rational point does not satisfy weak approximation with Brauer-Manin obstruction off $S,$ does not $X_L$ satisfy weak approximation with Brauer-Manin obstruction off $S_L?$	
	%\end{qu}
	\begin{qu}\label{question on HP}		
		Assume that the varieties $X$ and $X_L$ are counterexamples to the Hasse principle. If the failure of the Hasse principle of $X$ is explained by the Brauer-Manin obstruction, must the failure of the Hasse principle of $X_L$ also be explained by the Brauer-Manin obstruction?		
		%	Assume that the varieties $X$ and $X_L$ violate the Hasse principle over $K$ and $L$ respectively. If $X$ satisfies the Hasse principle with Brauer-Manin obstruction over $K$ (i.e. the failure of the Hasse principle of $X$ can be explained by the Brauer-Manin obstruction: $X(\AA_K)^\Br=\emptyset$),  must $X_L$ also satisfy the Hasse principle with Brauer-Manin obstruction over $L?$
	\end{qu}

	\subsection{Main results} In this paper, we will construct smooth, projective, and geometrically connected surfaces to give negative answers to Questions \ref{Questions}.
	
	\subsubsection{A negative answer to Question \ref{question on WA1}}
	For any number field $K,$ 	
	assuming Stoll's conjecture, Liang\cite{Li18} found a quadratic extension $L,$ and constructed	
	a $3$-fold to give a negative answer to Question \ref{question on WA1}.  When $L=\QQ(\sqrt{5})$ and $K=\QQ,$ using the construction method, he gave an unconditional example with explicit equations  in loc. cit.  The author \cite{Wu21} generalized his argument to any nontrivial extension of number fields. The varieties constructed there, are $3$-folds. In this paper, we will prove the same statement for smooth, projective, and geometrically connected surfaces.

	For any nontrivial extension of number fields $L/K,$ assuming Stoll's conjecture, we have the following theorem to give a negative answer to Question \ref{question on WA1}.
	\begin{thm}[Theorem \ref{theorem main result: non-invariance of weak approximation with BMO}]
		For any nontrivial extension of number fields $L/K,$ assuming Stoll's conjecture, there exists a smooth, projective, and geometrically connected surface  $X$ defined over $K$ such that
		\begin{itemize}
			\item the surface $X$ has a $K$-rational point, and satisfies weak approximation with Brauer-Manin obstruction %, \'etale Brauer-Manin obstruction
			off $\infty_K,$
			\item the surface $X_L$ does not satisfy weak approximation with Brauer-Manin obstruction %, \'etale Brauer-Manin obstruction
			off $T$ for any finite subset $T\subset \Omega_L.$ 
		\end{itemize}
	\end{thm}

	When $K=\QQ$ and $L=\QQ(i),$ using the construction method given in Theorem \ref{theorem main result: non-invariance of weak approximation with BMO}, we give an explicit unconditional example  in Subsection \ref{subsection main example1}. The smooth, projective, and geometrically connected surface  $X$ is defined by the following equations: 
	\begin{equation*}
		\begin{cases}
			(w_0w_2+w_1^2+16w_2^2)(x_0^2+x_1^2-x_2^2)+(w_0w_1+w_1w_2)(x_0^2-x_1^2)=0\\ 
			w_1^2w_2=w_0^3-16w_2^3
		\end{cases}
	\end{equation*}
	in $\PP^2\times \PP^2$ with bi-homogeneous coordinates $(w_0:w_1:w_2)\times(x_0:x_1:x_2).$
	
	% Although for these two fields, using the method given in \cite[Theorem 4.5]{Li18},  man can construct an explicit unconditional example having the properties of Theorem \ref{theorem main result: non-invariance of weak approximation with BMO}, the method that we use is different.
	
	%	Let $\zeta_7$ be a primitive $7$-th root of unity.
	%	For $K=\QQ$ and $L=\QQ(\zeta_7+\zeta_7^{-1}),$ based on the method given in Theorem \ref{theorem main result: non-invariance of weak approximation with BMO}, we give an explicit unconditional example  in Subsection \ref{subsection main example1}.  The method given in \cite[Theorem 4.5]{Li18} no longer works. The $3$-fold $X$ is a smooth compactification of the following $3$-dimensional affine subvariety given by equations 
	%	\begin{equation*}
	%		\begin{cases}
	%			
	%			
	%			
	%			
	%			
	%			
	%			
	%			
	%			y^2-73z^2=(1-x^2)(x^2-73)(x'-4)^2+(99x^2+1)(\frac{5428}{5329}x^2+\frac{1}{5329})\\ 
	%			{y'}^2={x'}^3-16
	%		\end{cases}
	%	\end{equation*}
	%	in $(x,y,z,x',y')\in \AA^5.$
	%	
	\subsubsection{A negative answer to Question \ref{question on HP}} 
	For any number field $K,$ suppose that Stoll's conjecture holds. Assuming some conditions on the nontrivial
	extension $L$ over $K,$ the author \cite{Wu21} constructed	
	a $3$-fold to give a negative answer to Question \ref{question on WA1}. Unconditional examples with explicit equations were given  in loc. cit. The varieties constructed there, are $3$-folds. In this paper, we will prove the same statement for smooth, projective, and geometrically connected surfaces.
	
	For any nontrivial extension of number fields $L/K,$ assuming Stoll's conjecture, we have the following theorem to give a negative answer to  Question \ref{question on HP}.
	
	\begin{thm}[Theorem \ref{theorem main result: non-invariance of the Hasse principle with BMO for odd degree}]
		For any nontrivial extension of number fields $L/K,$ assuming Stoll's conjecture, there exists a smooth, projective, and geometrically connected surface  $X$ defined over $K$ such that
		\begin{itemize} 
			\item the surface $X$ is a counterexample to the Hasse principle, and its failure of the Hasse principle is explained by the Brauer-Manin obstruction,
			\item the surface $X_L$ is a counterexample to the Hasse principle, but its failure of the Hasse principle cannot be explained by the Brauer-Manin obstruction.
		\end{itemize}
	\end{thm}

	When $K=\QQ$ and $L=\QQ(i),$ using the construction method given in Theorem \ref{theorem main result: non-invariance of the Hasse principle with BMO for odd degree}, we give an explicit unconditional example in Subsection \ref{subsection main example2}. The smooth, projective, and geometrically connected surface  $X$ is defined by the following two equations: 
	\begin{equation*}
		\begin{cases}
			(w_0w_2+w_1^2+16w_2^2)(x_0^2-41x_1^2)(x_0^2-3x_1^2)(x_0^2-123x_1^2)(y_0^2-13y_1^2)(y_0^3-41y_1^3)\\
			+(w_0w_1+w_1w_2)(x_0^2-17x_1^2)(x_0^2-13x_1^2)(x_0^2-221x_1^2)(y_0^2-53y_1^2)(y_0^3-53y_1^3)=0\\ 
			w_1^2w_2=w_0^3-16w_2^3
		\end{cases}
	\end{equation*}
	in $\PP^2\times \PP^1\times \PP^1$ with tri-homogeneous coordinates $(w_0:w_1:w_2)\times(x_0:x_1)\times(y_0:y_1).$

	\subsubsection{Main ideas behind our constructions in the proof of theorems}
	Let $L/K$ be a nontrivial extension of number fields. 
	We find a smooth, projective, and geometrically connected curve $C$ such that $C(K)$ and $C(L)$ are both finite, nonempty, and that $C(K)\neq C(L).$ Then we construct  a pencil of curves parametrized by the curve $C\colon$ $\beta \colon X\to C$ such that the fiber of each $C(K)$ point is isomorphic to one given curve denoted by $C_\infty,$ and that the fiber of each $C(L)\backslash C(K)$ point is isomorphic to another given curve denoted by $C_0.$ By combining some fibration arguments with the functoriality of Brauer-Manin pairing, the arithmetic properties of $C_\infty$ and $C_0$ will determine the arithmetic properties of $X.$ We carefully choose the curves $C_\infty$ and $C_0$ to meet the needs of theorems.

	\section{Notation and preliminaries}
	Let $K$ be a number field, and let $\Ocal_K$ be the ring of its integers. Let $\Omega_K$ be the set of all nontrivial places of $K.$ Let $\infty_K\subset \Omega_K$ be the subset of all archimedean places, and let $\Omega_K^f=\Omega_K\backslash \infty_K.$ Let $\infty_K^r\subset \infty_K$ be the subset of all real places, and let $2_K\subset \Omega_K$ be the subset of all $2$-adic places. 
	For $v\in \Omega_K,$ let $K_v$ be the completion of $K$ at  $v.$ For $v\in \infty_K^r,$ let $\tau_v\colon K\hookrightarrow \RR$ be the embedding of $K$ into its completion.	
	Given a finite subset $S\subset \Omega_K,$ let $\AA_K$ (respectively $\AA_K^S$) be the ring of ad\`eles (ad\`eles without $S$ components) of $K.$ We say that an element is a \defi{prime element}, if the ideal generated by this element is a prime ideal. For a prime element $p\in \Ocal_K,$ we denote %its prime ideal by $\pfr_a,$ and denote 
	its associated place by $v_p.$ We fix an algebraic closure $\overline{K}$ of $K,$ and let $\Gamma_K=\Gal(\overline{K}/K).$ We always assume that a field $L$ is a finite extension of $K.$ Let $S_L\subset \Omega_L$ be the subset of all places above $S.$

	In this paper, a \defi{$K$-scheme} will mean a reduced, separated scheme of finite type over $K,$ and all geometric objects are $K$-schemes. A \defi{$K$-curve}  will mean a proper $K$-scheme such that every irreducible components are of dimension one. In particular, a $K$-curve may have more than one irreducible component, and may have singular points. We say that a $K$-scheme is a \defi{$K$-variety} if it is geometrically integral. Be cautious that in our definition, a integral $K$-scheme may be not a variety, i.e. it may have multiple geometrically irreducible components. Given a proper $K$-scheme $X,$ if $X(\AA_K)\neq \emptyset,$ let $pr^S\colon X(\AA_K)\to X(\AA_K^S)$ be the projection induced by the natural projection $pr^S\colon\AA_K\to \AA_K^S.$ All cohomology groups in this paper are Galois or \'etale cohomology groups, and let $\Br(X)=H^2_{\et}(X,\GG_{m}).$

	By combining the \v{C}ebotarev density theorem with global class field theory, we have the following lemma to choose prime elements. This lemma is a generalization of Dirichlet's theorem on arithmetic progressions.

	\begin{lemma}\label{lemma dirichlet}
		Given  an extension of number fields $L/K,$ let $\Ifr\subset\Ocal_K$ be a proper nonzero ideal. Let $x\in \Ocal_K.$ Suppose that the image of $x$ in $\Ocal_K/\Ifr$ is invertible.  Then there exists a prime element $p\in \Ocal_K$ such that  
		\begin{enumerate}{
				\item\label{dirichlet condition 1}  $p\equiv x \mod \Ifr,$
				\item\label{dirichlet condition 2}  $\tau_v(p)>0$ for all $v\in \infty_K^r,$
				\item\label{dirichlet condition 3}  additionally, if $x=1,$ then $p$ splits completely in $L.$}
		\end{enumerate}
		And the set of places associated to such prime elements has positive density. 
	\end{lemma}

	\begin{proof}		
		Let $\mfr_{\infty}$ be the product of all places in $\infty_K^r,$ and let $\mfr=\Ifr \mfr_{\infty}$ be a modulus of $K.$ Let $K_\mfr$ be the ray class field of modulus $\mfr.$ %Let $I$ be the group of fractional ideals of $K,$ and 
		Let $I_\mfr$ be the group of fractional ideals that are prime to  $\Ifr.$ Let $P_\mfr\subset I_\mfr$ be the subgroup of principal ideals generated by some $a\in K^\times$ with $a\equiv 1 \mod \Ifr$ and $\tau_v(a)>0$ for all $v\in \infty_K^r.$ Then by Artin reciprocity law
		(cf. \cite[Theorem 7.1 and Corollary 7.2]{Ne99}), the classical Artin homomorphism $\theta$ gives an exact sequence:
		$$0\to P_\mfr\hookrightarrow I_\mfr \stackrel{\theta}\to \Gal(K_\mfr/K)\to 0.$$
		%Let $(x)$ be the ideal generated by $x.$ For the image of $x$ in $\Ocal_K/\Ifr$ is invertible, we have $(x)\in I_\mfr.$ Let $\langle\theta((x))\rangle\subset\Gal(K_\mfr/K) $ be the subgroup generated by $\theta((x)).$
		By the generalized Dirichlet density theorem (cf. \cite[Theorem 13.2]{Ne99}), the set of places associated to the prime elements satisfying conditions (\ref{dirichlet condition 1}) and (\ref{dirichlet condition 2}), has density $1 /[K_\mfr:K].$
		%$\sharp\langle\theta((x))\rangle /[K_\mfr:K].$
		Let $M$ be a smallest Galois extension of $K$ containing $L,$ then a place of $K$ splits completely in $L$ if and only if it  splits completely in $M.$ Let $MK_\mfr$ be a composition field of $M$ and $K_\mfr.$ If $x=1,$ then by the \v{C}ebotarev density theorem
		(cf. \cite[Theorem 13.4]{Ne99}), the set of places associated to the prime elements satisfying all these conditions (\ref{dirichlet condition 1}), (\ref{dirichlet condition 2}) and (\ref{dirichlet condition 3}), has density $1/[MK_\mfr:K].$
	\end{proof}

	\subsection{Hilbert symbol}

	For $a,b\in K_v^\times$ and $v\in \Omega_K,$ we use Hilbert symbol $(a,b)_v\in \{\pm 1\}.$  By definition, $(a,b)_v=1$ if and only if
	the curve defined over $K_v$ by the equation $x_0^2-ax_1^2-bx_2^2=0$ in $\PP^2$ with homogeneous coordinates $(w_0:w_1:w_2),$ has a $K_v$-point. %The Hilbert symbol gives a symmetric bilinear form on $K_v^\times/K_v^{\times 2}$ with value in $\ZZ/2\ZZ,$ cf. \cite[Chapter XIV, Proposition 7]{Se79}. And this bilinear form is nondegenerate, cf. \cite[Chapter XIV, Corollary 7]{Se79}. 

	\section{Stoll's conjecture for curves}

	For a smooth, projective, and geometrically connected curve $C$ defined over a number field $K,$ if the Tate-Shafarevich group and the rational points set of its Jacobian are both finite, then by combining the Cassels-Tate pairing with the Brauer evaluation pairing, Skorobogatov \cite[Chapter 6.2]{Sk01} and Scharaschkin \cite{Sc99} independently observed that $C(K)=pr^{\infty_K}(C(\AA_K)^\Br).$ In particular, if this curve $C$ is a counterexample to the Hasse principle, 
	then this failure can be explained by the Brauer-Manin obstruction. Stoll \cite[Theorem 8.6]{St07} generalized this observation. Furthermore, he \cite[Conjecture 9.1]{St07} made the following conjecture. 
	
	%		Given a curve $C$ defined over a number field $K,$ let $C(\AA_K)_\bullet=\prod_{v\in \infty_K}\{$connected components of $C(K_v)\}\times C(\AA_K^f).$ The product topology of $\prod_{v\in \infty_K}\{$connected components of $C(K_v)\}$
	%	with discrete topology and $C(\AA_K^f)$ with adelic topology, gives a topology for $C(\AA_K)_\bullet.$  For any $A\in\Br(C),$ and any $v\in\infty_K,$ the evaluation of $A$ on each connected component of $C(K_v)$ is constant. So, the notation $C(\AA_K)_\bullet^{\Br}$ makes sense.	
	%	

	\begin{conjecture}\cite[Conjecture 9.1]{St07}\label{conjecture Stoll}
		For any smooth, projective, and geometrically connected curve $C$ defined over a number field $K,$  the set $C(K)$ is dense in $pr^{\infty_K}(C(\AA_K)^\Br).$ In particular, if $C(K)$ is finite, then $C(K)=pr^{\infty_K}(C(\AA_K)^\Br).$
	\end{conjecture}

	\begin{remark}
		It is well known that for an elliptic curve over $\QQ$ of analytic rank $0,$ its Mordell-Weil group and Tate-Shafarevich group are both finite. By the dual sequence of Cassels-Tate, Conjecture \ref{conjecture Stoll} holds for this elliptic curve.
		%Original statement for weak approximation with Brauer-Manin obstruction is that $C(K)$ is dense in $C(\AA_K)_\bullet,$ here $C(\AA_K)_\bullet$ is obtained from $C(\AA_K)$ in the way of replacing archimedean places part by product of their sets of connected components with discrete topology. 
	\end{remark}

	%In this paper, we make the following assumptions for the given number field extension $L$ of $K.$
	The following definition and lemma have already been stated in the paper \cite{Wu21}. We give them below for the convenience of reading.
	\begin{definition}(\cite[Definition 4.0.3]{Wu21})\label{definition curve of type}
		Given a smooth, projective, and geometrically connected curve $C$ defined over a number field $K,$ let $L/K$ be a nontrivial extension of number fields. We say that a triple $(C,K,L)$  is of \defi{type $I$} if 
		\begin{itemize}
			\item the sets $C(K)$ and $C(L)$ are both finite and nonempty,
			\item $C(K)\neq C(L),$
			\item Stoll's conjecture \ref{conjecture Stoll} holds for the curve $C.$ 
		\end{itemize}
	\end{definition}

	\begin{lemma}(\cite[Lemma 4.0.4]{Wu21})\label{lemma stoll conjecture}
		Let $L/K$ be a nontrivial extension of number fields. Suppose that Conjecture \ref{conjecture Stoll} holds for all smooth, projective, and geometrically connected curves defined over $K.$ Then there exists a smooth, projective, and geometrically connected curve $C$ defined over $K$ such that the triple $(C,K,L)$ is of type $I.$
	\end{lemma}
	
	%\begin{proof}
	%	Since $L$ is a finite separable extension over $K,$ there exists a $\theta\in L$ such that $L=K(\theta).$ Let $f(x)$ be the minimal polynomial of $\theta.$ Let $n=\deg(f),$ then $n=[L:K]\geq 2.$ Let $\tilde{f}(w_0,w_1)$ be the homogenization of $f.$ We consider a curve $C$ defined over $K$ by a homogeneous equation: $w_2^{n+2}=\tilde{f}(w_0,w_1)(w_1^2-w_0^2)$ with homogeneous coordinates $(w_0:w_1:w_2)\in \PP^2.$ For the polynomials $f(x)$ and $x^2-1$ are separable and coprime in $K[x],$ the curve $C$ is smooth, projective, and geometrically connected. By genus formula for a plane curve, the genus of  $C$ equals $g(C)=n(n+1)/2>1.$ By Faltings's theorem, the sets $C(K)$ and $C(L)$ are both finite. It is easy to check that $(w_0:w_1:w_2)=(1:1:0)\in C(K)$ and $(\theta:1:0)\in C(L)\backslash C(K).$ By the assumption that Conjecture \ref{conjecture Stoll} holds for all smooth, projective, and geometrically connected curves over $K,$ we have that the triple $(C,K,L)$ is of type $I.$
	%\end{proof}

	The following lemma is a strong form of \cite[Lemma 6.1.3]{Wu21}. It will be used to choose a dominant morphism from a given curve to $\PP^1.$

	\begin{lemma}\label{lemma choose base change morphism}
		Let $L/K$ be a nontrivial extension of number fields. Given a smooth, projective, and geometrically connected curve $C$ defined over $K,$ suppose that the triple $(C,K,L)$ is of type $I$ (Definition \ref{definition curve of type}). For any finite $K$-subscheme $R\subset \PP^1,$  there exists a dominant $K$-morphism $\gamma\colon  C\to \PP^1$ such that 
		\begin{itemize}{
				\item  $\gamma(C(K))=\{\infty\}\subset \PP^1(K),$
				\item $\gamma(C(L)\backslash C(K))=\{0\}\subset \PP^1(K),$
				%\item there exists some $P\in C(L)\backslash C(K)$ such that $\gamma(P)=0\in \PP^1(K),$ 
				%	\item $\gamma$ is \'etale over $R\bigcup \{0,\infty\}.$    
				\item $\gamma$ is \'etale over $R.$    }
		\end{itemize}
	\end{lemma}
	
	\begin{proof}
		The proof is along the same idea as the proof of \cite[Lemma 6.1.3]{Wu21}, where the statement was shown for $R\subset \PP^1\backslash\{0,\infty\}.$
		We will put one more condition for choosing a rational function. Let $K(C)$ be the function field of $C.$ % Let $M$ be the smallest Galois extension containing $L$ of $K,$ and let $\Gal(M/K)$ be the Galois group.
		For $C(K)$ and $C(L)$ are both finite nonempty, and $C(K)\neq C(L),$ by Riemann-Roch theorem, we can choose a rational function $\phi\in K(C)^\times\backslash K^\times$ such that
		\begin{itemize}
			\item the set of its poles contains $C(K),$
			\item the set of its zeros contains $C(L)\backslash C(K),$
			\item all poles and zeros are of multiplicity one.
		\end{itemize}
		Then this rational function $\phi$ gives a dominant $K$-morphism $\gamma_0\colon C\to \PP^1$ such that 
		\begin{itemize}
			\item $\gamma_0(C(L)\backslash C(K))=\{0\}\subset \PP^1(K),$
			\item $\gamma_0(C(K))=\{\infty\}\subset \PP^1(K),$
			\item $\gamma_0$ is \'etale over $\{0,\infty\}.$
		\end{itemize}
		Then the branch locus of $\gamma_0$ is finite and contained in $\PP^1\backslash\{0,\infty\}.$
		We can choose an automorphism $\varphi_{\lambda_0}\colon \PP^1\to \PP^1, (u:v)\mapsto (\lambda_0 u:v)$ with $\lambda_0\in K^\times$ such that the branch locus of $\gamma_0$ has no intersection with $\varphi_{\lambda_0}(R).$ Let $\gamma= (\varphi_{\lambda_0})^{-1}\circ\gamma_0.$ Then the morphism $\gamma$ is \'etale over $R,$ and satisfies other conditions. 
	\end{proof}

	\section{Main results}

	In this section, we will construct smooth, projective, and geometrically connected surfaces to give negative answers to Questions \ref{Questions}.

	\subsection{Non-invariance of weak approximation with Brauer-Manin obstruction for surfaces}
	For any number field $K,$ assuming Conjecture \ref{conjecture Stoll}, Liang \cite[Theorem 4.5]{Li18} found a quadratic extension $L,$ and constructed  a $3$-fold to give a negative answer to Question \ref{question on WA1}. The author \cite[Theorem 6.2.1]{Wu21} generalized his result to any nontrivial extension of number fields. Although the strategies of these two papers are different, the methods used there are combining the arithmetic properties of Ch\^atelet surfaces with a construction method from Poonen \cite{Po10}. Thus 
	the varieties constructed there, are $3$-folds.  
	For any extension of number fields $L/K,$ assuming Conjecture \ref{conjecture Stoll}, in this subsection, we will construct a smooth, projective, and geometrically connected surface to give a negative answer to Question \ref{question on WA1}. The method that we will use, is to combine some fibration lemmas with the arithmetic properties of curves, whose irreducible components are projective lines. 
	
	\subsubsection{Preparation Lemmas} We state the following lemmas, which will be used for the proof of Theorem \ref{theorem main result: non-invariance of weak approximation with BMO}.

	The following fibration lemma has already been stated in the paper \cite{Wu21}. We give them below for the convenience of reading.
	%The following two fibration Lemmas  will be used for the proof of our main theorems. In this subsection, we will give two fibration lemmas. 
	\begin{lemma}(\cite[Lemma 6.1.1]{Wu21})\label{lemma fiber criterion for wabm}
		Let $K$ be a number field, and let $S\subset \Omega_K$ be a finite subset.  Let $f\colon X\to Y$ be a $K$-morphism of proper $K$-varieties $X$ and $Y$. Suppose that
		\begin{enumerate}{
				\item\label{fiber criterion for wabm condition 1}  the set $Y(K)$ is finite,
				\item\label{fiber criterion for wabm condition 2}  the variety $Y$ satisfies weak approximation with Brauer-Manin obstruction off $S,$
				\item\label{fiber criterion for wabm condition 3}  for any $P\in Y(K),$ the fiber $X_P$ of $f$ over $P$ satisfies weak approximation off $S.$}
		\end{enumerate}
		Then the variety $X$ satisfies weak approximation with Brauer-Manin obstruction off $S.$
	\end{lemma}
	
	%\begin{proof}
	%	For any finite subset $S'\subset\Omega_K\backslash S,$ take an open subset $N=\prod_{v\in S'}U_v\times \prod_{v\notin S'}X(K_v)\subset X(\AA_K)$ such that $N\cap X(\AA_K)^{\rm Br}\neq \emptyset.$ Let $M=\prod_{v\in S'}f(U_v)\times \prod_{v\notin S'}f(X(K_v)),$ then by the functoriality of Brauer-Manin pairing, $M\cap Y(\AA_K)^{\rm Br}\neq \emptyset.$ By Assumptions (\ref{fiber criterion for wabm condition 1}) and (\ref{fiber criterion for wabm condition 2}), we have $Y(K)= pr^S(Y(\AA_K)^{\rm Br}).$ So there exists $P_0\in pr^S(M)\cap Y(K).$ Consider the fiber $X_{P_0}.$ Let $L=\prod_{v\in S'} [X_{P_0}(K_v)\cap U_v]\times \prod_{v\notin S'\cup S} X_{P_0}(K_v),$ then it is a nonempty open subset of $X_{P_0}(\AA_K^S).$  By Assumption (\ref{fiber criterion for wabm condition 3}), there exists $Q_0\in L\cap X_{P_0}(K).$ So $Q_0\in X(K)\cap N,$ which implies that $X$ satisfies weak approximation with Brauer-Manin obstruction off $S.$
	%\end{proof}

	The following fibration lemma can be viewed as a modification of \cite[Lemma 6.1.2]{Wu21} to fit into our context.
	
	\begin{lemma}\label{lemma fiber criterion for not wabm}
		Let $K$ be a number field, and let $S\subset \Omega_K$ be a finite subset.  Let $f\colon X\to Y$ be a $K$-morphism of proper $K$-varieties $X$ and $Y$. We assume that
		\begin{enumerate}{
				\item\label{fiber criterion for not wabm condition 1}  the set $Y(K)$ is finite,
				\item\label{fiber criterion for not wabm condition 2} there exists some $P\in Y(K)$ such that the fiber $X_P$ of $f$ over $P$ does not satisfy weak approximation with Brauer-Manin obstruction off $S.$   }
		\end{enumerate}
		Then the variety $X$ does not satisfy weak approximation with Brauer-Manin obstruction off $S.$
	\end{lemma}
	
	\begin{proof}
		By Assumption (\ref{fiber criterion for not wabm condition 2}), take a $P_0\in Y(K)$ such that the fiber $X_{P_0}$ does not satisfy weak approximation with Brauer-Manin obstruction off $S.$ Then there exist a finite nonempty subset $S'\subset\Omega_K\backslash S$ and a nonempty open subset 	
		$L=\prod_{v\in S'}U_v\times \prod_{v\notin S'}X_{P_0}(K_v)\subset X_{P_0}(\AA_K)$ such that $L\cap X_{P_0}(\AA_K)^{\Br}\neq \emptyset,$ but that $L\cap X_{P_0}(K)=\emptyset.$
		By Assumption (\ref{fiber criterion for not wabm condition 1}), the set $Y(K)$ is finite, so we can take a Zariski open subset $V_{P_0}\subset Y$ such that $V_{P_0}(K)=\{P_0\}.$ For any $v\in S',$ since $U_v$ is open in $X_{P_0}(K_v)\subset f^{-1}(V_{P_0})(K_v),$ we can take an open subset $W_v$ of $f^{-1}(V_{P_0})(K_v)$ such that $W_v\cap X_{P_0}(K_v)=U_v.$ 	
		Consider the open subset $N=\prod_{v\in S'}W_v\times \prod_{v\notin S'}X(K_v)\subset X(\AA_K),$ then $L\subset N.$ By the functoriality of Brauer-Manin pairing, we have $X_{P_0}(\AA_K)^{\Br}\subset  X(\AA_K)^{\Br}.$ So the set $N\cap X(\AA_K)^{\Br}\supset L\cap X_{P_0}(\AA_K)^{\Br},$ is nonempty. But $N\cap X(K)=N\cap X_{P_0}(K)=L\cap X_{P_0}(K)
		=\emptyset,$ which implies that $X$ does not satisfy weak approximation with Brauer-Manin obstruction off $S.$
	\end{proof}

	The following lemma states that a $K$-scheme with multiple geometrically irreducible components will violate weak approximation.
	\begin{lemma}\label{lemma prevariety not satifying WA}
		Let $K$ be a number field, and let $S\subset \Omega_K$ be a finite subset. Let $X$ be a $K$-scheme, which is not a $K$-variety, i.e. it has multiple geometrically irreducible components.  We assume $\prod_{v\in \Omega_K}X(K_v)\neq \emptyset,$ then the variety
		$X$ does not satisfy weak approximation off $S.$
	\end{lemma}
	
	\begin{proof}
		%By assumption, we can take a irreducible component $X_1$ of $X$ such that $\dim X_1\geq 1.$ Take another irredu
		Let $X^0$ be the smooth locus of $X.$ Claim that $X^0\subset X$ is an open dense subscheme.
		We prove the claim first. For $X$ is reduced and $K$ is of characteristic $0,$ the scheme $X$ is geometrically reduced. For any geometrically irreducible component of $X,$  by \cite[Chapter II. Corollary 8.16]{Ha97}, its smooth locus is open dense in this geometrically irreducible component. So the claim follows. From this claim, we have $X$ and $X^0$ have the same number of geometrically irreducible components.	
		
		By assumption that $X$ has multiple geometrically irreducible components, let $X_1^0$ and $X_2^0$ be two different geometrically irreducible components of $X^0,$ defined over the number fields $K_1$ and $K_2$ respectively. By Lang-Weil estimate \cite{LW54}, the varieties $X_1^0$ and $X_2^0$ have local points for almost all places of $K_1$ and $K_2$ respectively. By the \v{C}ebotarev density theorem, we can take two different places $v_1,v_2\in \Omega_K^f\backslash S$ such that $v_1,v_2$ split in $K_1$ and also in $K_2,$ and that $X_1^0(K_{v_1})\neq\emptyset$ and $X_2^0(K_{v_2})\neq\emptyset.$  For  $\prod_{v\in \Omega_K}X(K_v)\neq \emptyset,$ we consider a nonempty open subset $L=X_1^0(K_{v_1})\times X_2^0(K_{v_2})\times\prod_{v\in \Omega_K\backslash \{v_1,v_2\}}X(K_v)\subset \prod_{v\in \Omega_K}X(K_v).$ For $X^0$ is smooth, and the varieties $X_1^0,~X_2^0$ are different geometrically irreducible components, we have $X_1^0(K_{v_1})\cap X_2^0(K_{v_1})=\emptyset,$ which implies $X(K)\cap L=\emptyset.$ Hence $X$ does not satisfy weak approximation off $S.$
	\end{proof}
	
	The following two lemmas state that two projective lines meeting at one point will violate weak approximation with Brauer-Manin obstruction.
	
	\begin{lemma}\label{lemma two projective lines Brauer group}
		Let $C$ be a curve defined over a number field $K$ by a homogeneous equation:
		$x_0^2-x_1^2=0$ in $\PP^2$ with homogeneous coordinates $(x_0:x_1:x_2).$ Then the natural restriction map $\Br(K)\to \Br(C),$ is an isomorphism. 
	\end{lemma}
	
	\begin{proof}
		Let $C_1$ and $C_2$ be two irreducible components of $C.$ Let $i_1,$ $i_2$ and $i_3$ be the natural embeddings of $C_1,$  $C_2$ and $C_1\cap C_2$ in $C$ respectively. Then we have the following sequence of \'etale sheaves on $C\colon$
		$$0\to \Ocal_C\to i_{1*}\Ocal_{C_1}\oplus i_{2*}\Ocal_{C_2}\to i_{3*}\Ocal_{C_1\cap C_2}\to 0,$$
		where the map $i_{2*}\Ocal_{C_2}\to i_{3*}\Ocal_{C_1\cap C_2}$ is the opposite of the restriction map, and other maps are canonical restriction maps. By checking the exactness of this sequence at each geometric point of $C,$ and \cite[Chapter II. Theorem 2.15]{Mi80}, it is exact.
		It gives rise to an exact sequence of \'etale sheaves on $C\colon$
		$$	0\to\GG_{m,C}\to i_{1*}\GG_{m,C_1}\oplus i_{2*}\GG_{m,C_2}\to i_{3*}\GG_{m,C_1\cap C_2}\to 0.$$
		For the intersection $C_1\cap C_2$ is a rational point, this sequence splits. Using \'etale cohomology, for any integer $n\geq 0,$ we have an exact sequence:
		$$0\to H^n_{\et}(C,\GG_{m})\to H^n_{\et}(C,i_{1*}\GG_{m,C_1}\oplus i_{2*}\GG_{m,C_2})\to H^n_{\et}(C,i_{3*}\GG_{m,C_1\cap C_2})\to 0.$$
		For $i_1,$ $i_2$ and $i_3$ are closed embeddings, by \cite[Chapter II. Corollary 3.6]{Mi80}, the functors $i_{1*},$ $i_{2*}$ and $i_{3*}$ are exact. Since $C_1$ and $C_2$ are isomorphic to $\PP^1,$ we have the following commutative diagram:
		$$ \xymatrix{
			0\ar[r]& H^n_{\et}(C,\GG_{m})\ar[r]\ar@{=}[d]& H^n_{\et}(C,i_{1*}\GG_{m,C_1}\oplus i_{2*}\GG_{m,C_2})\ar[r]\ar[d]^{\cong}& H^n_{\et}(C,i_{3*}\GG_{m,C_1\cap C_2})\ar[r]\ar[d]^{\cong}& 0\\
			0\ar[r]& H^n_{\et}(C,\GG_{m})\ar[r]& H^n_{\et}(\PP^1,\GG_{m})\oplus H^n_{\et}(\PP^1,\GG_{m})\ar[r]& H^n(\Gamma_K,\overline{K}^\times)\ar[r]& 0\\
		}$$
		with exact rows. By taking $n=2,$ we have an exact sequence: $$0\to \Br(C)\to \Br(K) \oplus \Br(K) \to \Br(K)\to 0.$$ So we have $\Br(K)\cong \Br(C).$	
	\end{proof}
	
	\begin{remark}
		In \cite{HS14}, Harpaz and Skorobogatov used another exact sequence of \'etale sheaves on $C$ (cf. Proposition 1.1 in loc. cit.) to calculate the Brauer group of $C.$ By easy computation, this lemma can be gotten from their Corollary 1.5 in loc. cit.
	\end{remark}

	\begin{lemma}\label{lemma two projective lines not satifying WA}
		Let $K$ be a number field, and let $S\subset \Omega_K$ be a finite subset. Let $C$ be a curve defined over $K$ by a homogeneous equation:
		$x_0^2-x_1^2=0$ in $\PP^2$ with homogeneous coordinates $(x_0:x_1:x_2).$ Then 
		the curve $C$ does not satisfy weak approximation with Brauer-Manin obstruction off $S.$
	\end{lemma}
	
	\begin{proof}
		For the curve $C$ has $K$-rational points and two irreducible components, by Lemma \ref{lemma prevariety not satifying WA}, it does not satisfy weak approximation off $S.$ By Lemma \ref{lemma two projective lines Brauer group}, we have $\Br(K)\cong \Br(C).$ So the curve $C$ does not satisfy weak approximation with Brauer-Manin obstruction off $S.$
	\end{proof}

	\begin{theorem}\label{theorem main result: non-invariance of weak approximation with BMO}
		For any nontrivial extension of number fields $L/K,$ assuming that Conjecture \ref{conjecture Stoll} holds over $K,$ there exists a smooth, projective, and geometrically connected surface  $X$ defined over $K$ such that
		\begin{itemize}
			\item the surface $X$ has a $K$-rational point, and satisfies weak approximation with Brauer-Manin obstruction %, \'etale Brauer-Manin obstruction
			off $\infty_K,$
			\item the surface $X_L$ does not satisfy weak approximation with Brauer-Manin obstruction %, \'etale Brauer-Manin obstruction
			off $T$ for any finite subset $T\subset \Omega_L.$ 
		\end{itemize}
	\end{theorem}
	
	\begin{proof}
		We will construct a smooth, projective, and geometrically connected surface $X.$ Let $C_\infty$ be a projective line defined over $K$ by a homogeneous equation:
		$x_0^2+x_1^2-x_2^2=0$ in $\PP^2$ with homogeneous coordinates $(x_0:x_1:x_2).$ Let $C_0$ be a curve defined over $K$ by a homogeneous equation:
		$x_0^2-x_1^2=0$ in $\PP^2$ with homogeneous coordinates $(x_0:x_1:x_2).$
		Let $(u_0:u_1)\times(x_0:x_1:x_2)$ be the coordinates of $\PP^1\times\PP^2,$ and let $s'=u_0(x_0^2+x_1^2-x_2^2)+u_1(x_0^2-x_1^2)\in \Gamma(\PP^1\times\PP^2,\Ocal(1,2)).$ Let $X'$ be the locus defined by $s'=0$ in $\PP^1\times\PP^2.$ For the curves $C_\infty$ and $C_0$ meet transversally, the locus $X'$ is smooth. Let $R$ be the locus over which the composition $ X'\hookrightarrow \PP^1\times\PP^2  \stackrel{pr_1}\to\PP^1$ is not smooth.  Then by \cite[Chapter III. Corollary 10.7]{Ha97}, it is finite over $K.$ 
		By the assumption that Conjecture \ref{conjecture Stoll} holds over $K,$ and Lemma \ref{lemma stoll conjecture}, we can take a smooth, projective, and geometrically connected curve $C$ defined over $K$ such that the triple $(C,K,L)$ is of type $I.$ By Lemma \ref{lemma choose base change morphism}, we can choose a $K$-morphism $\gamma\colon  C\to \PP^1$ such that $\gamma(C(L)\backslash C(K))=\{0\}\subset \PP^1(K),$ $\gamma(C(K))=\{\infty\}\subset \PP^1(K),$ and that $\gamma$ is \'etale over $R.$ 
		Let $B=C\times \PP^2,$ and let $(\gamma,id)\colon B\to\PP^1\times \PP^2.$  Let $\Lcal=(\gamma,id)^*\Ocal(1,2),$ and let $s=(\gamma,id)^* (s')\in \Gamma(B,\Lcal).$ Let $X$ be the zero locus of $s$ in $B.$
		For $\gamma$ is \'etale over the locus  $R,$ the surface $X$ is smooth.
		Since $X$ is defined by the support of the global section $s,$ it is an effective divisor. The invertible sheaf $\Lscr (X')$ on $\PP^1\times\PP^2$ is isomorphic to $\Ocal(1,2),$ which is a very ample sheaf on $\PP^1\times\PP^2.$ And $(\gamma,id)$ is a finite morphism, so the pull back of this ample sheaf is again ample, which implies that the invertible sheaf $\Lscr (X)$ on $C\times\PP^2$ is ample. By \cite[Chapter III. Corollary 7.9]{Ha97}, the surface $X$ is geometrically connected. So the surface $X$ is smooth, projective, and geometrically connected. 
		Let $\beta\colon X \hookrightarrow B=C\times \PP^2 \stackrel{pr_1}\to C$ be the composition morphism. 
		By our construction, we have the following Cartesian diagram:
			$$\xymatrix{
				X \ar@{^(->}[d]\ar[r] \ar@/_2.5pc/[dd]_{\beta} & X'  \ar@{^(->}[d]  \\
				C\times \PP^2 \ar^{pr_1}[d] \ar^{(\gamma,id)}[r]&  \PP^1\times \PP^2 \ar^{pr_1}[d]\\
				C \ar^{\gamma}[r] & \PP^1  
			}$$

		Next, we will check that the surface $X$ has the properties.
		
		We will show that  $X$ has a $K$-rational point. For any $P\in C(K),$ the fiber  $\beta^{-1}(P)\cong C_\infty.$ The projective line $C_\infty$ has a $K$-rational point, so the set $X(K)\neq \emptyset.$\\
		%For $\Br(C_\infty)/\Br(K)=0,$ according to \cite[Theorem B]{CTSSD87a,CTSSD87b},
		We will show that  $X$ satisfies weak approximation with Brauer-Manin obstruction %, \'etale Brauer-Manin obstruction 
		off $\infty_K.$ 
		Since the projective line $C_\infty$ satisfies weak approximation, also weak approximation off $\infty_K,$  we consider the morphism $\beta,$ then
		Assumption (\ref{fiber criterion for wabm condition 3}) of Lemma \ref{lemma fiber criterion for wabm} holds.
		Since Conjecture \ref{conjecture Stoll} holds for the curve $C,$ using Lemma \ref{lemma fiber criterion for wabm} for the morphism $\beta,$ the surface $X$ satisfies weak approximation with Brauer-Manin obstruction off $\infty_K.$ %, \'etals
		
		For any finite subset $T\subset \Omega_L,$ we will show that  $X_L$ does not satisfy weak approximation with Brauer-Manin obstruction %, \'etale Brauer-Manin obstruction 
		off $T.$ We take a point  $Q\in C(L)\backslash C(K),$ by the choice of the curve $C$ and morphism $\beta,$ the fiber $\beta^{-1}(Q)\cong C_{0L}.$
		By Lemma \ref{lemma two projective lines not satifying WA}, the curve $C_{0L}$ does not satisfy weak approximation with Brauer-Manin obstruction off $T\cup \infty_L.$ By Lemma \ref{lemma fiber criterion for not wabm}, the surface $X_L$ does not satisfy weak approximation with Brauer-Manin obstruction %, \'etale Brauer-Manin obstruction 
		off $T\cup \infty_L.$ So it does not satisfy weak approximation with Brauer-Manin obstruction %, \'etale Brauer-Manin obstruction 
		off $T.$
	\end{proof}

	\subsection{Non-invariance of the failures of the Hasse principle explained by the Brauer-Manin obstruction for surfaces}

	For an extension of number fields $L/K,$ assuming that  the degree $[L:K]$ is odd, or that the field $L$ has one real place, also assuming Conjecture \ref{conjecture Stoll}, the author  \cite[Theorem 6.3.1 and Theorem 6.3.2]{Wu21} constructed $3$-folds to give negative answers to Question \ref{question on HP}. The method used there is combining the arithmetic properties of Ch\^atelet surfaces with a construction method from Poonen \cite{Po10}. Thus 
	the varieties constructed there, are $3$-folds. For any extension of number fields $L/K,$ assuming Conjecture \ref{conjecture Stoll}, in this subsection, we will construct a smooth, projective, and geometrically connected surface to give a negative answer to Question \ref{question on HP}.

	\subsubsection{Preparation lemmas} We state the following lemmas, which will be used for Choosing curves.

	\begin{lemma}\label{lemma for two prime elements one is square for another}
		Let $K$ be a number field.  Let $p_1,p_2$ be two odd prime elements, and $v_{p_1}\neq v_{p_2}.$
		If $(p_1,p_2)_{v_{p_1}}=1,$ then $p_2\in K_{v_{p_1}}^{\times 2}.$ Otherwise, if $(p_1,p_2)_{v_{p_1}}=-1,$ then $p_2\notin K_{v_{p_1}}^{\times 2}.$
	\end{lemma}
	
	\begin{proof}
		Consider the case $(p_1,p_2)_{v_{p_1}}=1.$ By definition, the equation $x_0^2-p_1x_1^2-p_2x_2^2=0$  has a nontrivial solution in $K_{v_{p_1}}.$ Let $(x_0,x_1,x_2)=(a,b,c)$ be a primitive solution of this equation. By comparing the valuations, we have $ v_{p_1}(a)=v_{p_1}(c)=0.$ So $a^2-p_2c^2\equiv 0\mod p_1.$ For $p_1$ is an odd prime element, by Hensel's lemma, we have $p_2\in K_{v_{p_1}}^{\times 2}.$ This proves the first part of this lemma. If $p_2\in K_{v_{p_1}}^{\times 2},$ then $(p_1,p_2)_{v_{p_1}}=1,$ which implies the last argument.
	\end{proof}

	\begin{lemma}\label{lemma for square prime element}
		Let $K$ be a number field, and let $v\in \Omega_K^f.$ 
		Then there exists a proper nonzero ideal $\Ifr\subset\Ocal_K$ such that for any  $a\in \Ocal_K,$ if $a\equiv 1\mod \Ifr,$ then $a\in K_v^{\times 2}.$
	\end{lemma}
	
	\begin{proof}
		Let $p$ be the prime number such that $v|p$ in $K.$ Let $\Ifr$ be the ideal generated by $p^3.$
		Then by Hensel's lemma, we have  $1+ p^3 \Ocal_{K_v}\subset K_v^{\times 2},$ which implies this lemma.
	\end{proof}

	\begin{lemma}\label{lemma for hilbert prime element -1}
		Let $K$ be a number field. Let $p_1,p_2$ be two odd prime elements, and $v_{p_1}\neq v_{p_2}.$
		Let $\Ifr\subset \Ocal_K$ be the ideal generated by $p_1p_2.$
		Then there exists an element $x\in \Ocal_K$ such that
		\begin{itemize}
			\item the image of $x$ in $\Ocal_K/\Ifr$ is invertible,
			\item for any  $a\in \Ocal_K,$ if $a\equiv x\mod \Ifr,$ then $(p_1,a)_{v_{p_1}}=-1$ and $(p_2,a)_{v_{p_2}}=1.$
		\end{itemize}
	\end{lemma}
	
	\begin{proof}
		We take an element $\overline{x_1}\in (\Ocal_K/p_1)\backslash (\Ocal_K/p_1)^2,$ and let $x_1\in \Ocal_K$ be a lift of $\overline{x_1}.$ By Chinese remainder theorem, we choose an element $x\in \Ocal_K$ such that $x\equiv x_1\mod p_1$ and $x\equiv 1 \mod p_2.$ By the similar argument as in the proof of Lemma \ref{lemma for two prime elements one is square for another}, this element $x$ satisfies the conditions.
	\end{proof}

	\subsubsection{Choosing one curve with respect to an extension}\label{subsection choose a polynomial with respect an extension}
	In this subsubsection, we will choose one curve with some given arithmetic properties.
	Given an extension of number fields $L/K,$ by Lemmas \ref{lemma for square prime element} and \ref{lemma dirichlet}, we can choose an odd prime element $p_1\in \Ocal_K$ satisfying the following conditions:
	\begin{itemize}
		%\item $\tau_v(a)<0$ for all $v\in S\cap \infty_K,$
		\item $\tau_v(p_1)>0$ for all $v\in \infty_K^r,$
		\item $p_1\in K_v^{\times 2}$  for all $v\in 2_K,$ 
		\item $p_1$ splits in $L.$
	\end{itemize}
	By Lemmas \ref{lemma for two prime elements one is square for another}, \ref{lemma for square prime element} and \ref{lemma dirichlet}, we can choose an odd prime element $p_2\in \Ocal_K$ satisfying the following conditions:
	\begin{itemize}
		\item $(p_1,p_2)_{v_{p_1}}=1,$
		\item  $p_2$ splits in $L,$
		\item  $v_{p_2}\neq v_{p_1}.$
	\end{itemize}
	%Let $f(x_0,x_1)=(x_0^2-p_1x_1^2)(x_0^2-p_2x_1^2)(x_0^2-p_1p_2x_1^2)$ be a homogeneous polynomial in $K[x_0,x_1].$ Let $Z^f$ be the zero locus of $f$ in $\PP^1.$ It is a $0$-dimensional $K$-scheme. 

	Let $L'=L(\sqrt{p_1},\sqrt{p_2}).$ By Lemma \ref{lemma dirichlet}, we can choose an odd prime element $p_3\in \Ocal_K$ such that $v_{p_3}\notin\{v_{p_1},v_{p_2}\},$ and that $v_{p_3}$ splits in $L'.$  %Let $F_1=K(\sqrt{c})$ and $F_2=K(\sqrt[3]{c}).$ 
	Let $f(x_0,x_1;y_0,y_1)=(x_0^2-p_1x_1^2)(x_0^2-p_2x_1^2)(x_0^2-p_1p_2x_1^2)(y_0^2-p_3y_1^2)(y_0^3-p_3y_1^3)$ be a bi-homogeneous polynomial, and let $Z^f$ be the zero locus of $f$ in $\PP^1\times\PP^1$ with bi-homogeneous coordinates $(x_0:x_1)\times(y_0:y_1).$ With the notation, we have the following lemmas.

	% Let $\Spec F_1$ (respectively $\Spec F_1$) be the $0$-dimensional $K$-scheme defined by a homogeneous equation:
	%$y_0^2-cy_1^2=0$ (respectively $y_0^3-cy_1^3=0$) with homogeneous coordinates $(y_0:y_1)\in \PP^1.$ Consider the curve $C_0=(Z^f\times \PP^1)\cup (\PP^1\times \Spec F_1)\cup (\PP^1\times \Spec F_2)$ in $ \PP^1\times\PP^1$ with bi-homogeneous coordinates $(x_0,x_1)\times (y_0,y_1).$  With the notation, by \cite[Proposition 3.1]{HS14}, we have the following lemma.

	\begin{lemma}\label{lemma zero dimension scheme violate HP}
		Let $Z^f\subset \PP^1\times\PP^1$ be the zero locus defined over $K$  by the bi-homogeneous polynomial $f(x_0,x_1;y_0,y_1).$ Then the curves $Z^f$ and $Z^f_L$ violate the Hasse principle.
	\end{lemma}
	
	\begin{proof}
		By the condition that the prime elements $p_1,p_2$ and $p_3$ split in $L,$ the set $Z^f(K)=Z^f(L)= \emptyset.$
		It will be suffice to prove that for any $v\in \Omega_K,$ the equation $(x_0^2-p_1x_1^2)(x_0^2-p_2x_1^2)(x_0^2-p_1p_2x_1^2)=0$ has a $K_v$-solution in $\PP^1$ with homogeneous coordinates $(x_0:x_1).$ 
		
		Suppose that $v\in \infty_K\cup 2_K.$  Then, by the choice of $p_1,$ we have $p_1\in K_v^{\times 2},$ so the equation $x_0^2-p_1x_1^2=0$ has a $K_v$-solution in $\PP^1.$\\
		Suppose that $v=v_{p_1}.$ Then, by the choice of $p_2,$ we have $(p_1,p_2)_{v_{p_1}}=1.$ By Lemma \ref{lemma for two prime elements one is square for another}, we have $p_2\in K_{v_{p_1}}^{\times 2}.$	
		Hence the equation $x_0^2-p_2x_1^2=0$ has a $K_v$-solution in $\PP^1.$\\
		Suppose that $v=v_{p_2}.$ Using the product formula $\prod_{v\in \Omega_K}(p_1,p_2)_v=1,$ we have $(p_1,p_2)_{v_{p_2}}=1.$ By Lemma \ref{lemma for two prime elements one is square for another}, we have $p_1\in K_{v_{p_2}}^{\times 2}.$	
		Hence the equation $x_0^2-p_1x_1^2=0$ has a $K_v$-solution in $\PP^1.$\\
		Suppose that  $v\in \Omega_K\backslash (\infty_K\cup 2_K\cup\{v_{p_1}, v_{p_2}\}),$ then,
		by the quadratic reciprocity law, at least one of equations: $x_0^2-p_1x_1^2=0,~x_0^2-p_2x_1^2=0,~x_0^2-p_1p_2x_1^2=0,$ has a $K_v$-solution in $\PP^1.$\\
		So $Z^f(\AA_K)\neq \emptyset.$	
	\end{proof}

	\begin{lemma}\label{lemma zero dimension scheme trival Brauer group pairing}
		The natural restriction map  $\Br (L)\to \Br (Z^f_L),$ is an isomorphism.
	\end{lemma}
	
	\begin{proof}
		Let $C_1$ (respectively $C_2$) be the locus defined over $L$ by the equation $(x_0^2-p_1x_1^2)(x_0^2-p_2x_1^2)(x_0^2-p_1p_2x_1^2)=0$ (respectively $(y_0^2-p_3y_1^2)(y_0^3-p_3y_1^3)=0$) in $\PP^1\times\PP^1$ with bi-homogeneous coordinates $(x_0:x_1)\times(y_0:y_1).$
		Then $C_1$ and $C_2$ are smooth curves in $Z^f_L,$ and $Z^f_L=C_1 \cup C_2.$
		Let $i_1,$ $i_2$ and $i_3$ be the natural embeddings of $C_1,$  $C_2$ and $C_1\cap C_2$ in $C$ respectively. Similar to the proof of Lemma \ref{lemma two projective lines Brauer group}, we have the following exact sequence of \'etale sheaves on $Z^f_L\colon$
		$$0\to \Ocal_{Z^f_L}\to i_{1*}\Ocal_{C_1}\oplus i_{2*}\Ocal_{C_2}\to i_{3*}\Ocal_{C_1\cap C_2}\to 0,$$
		where the map $i_{2*}\Ocal_{C_2}\to i_{3*}\Ocal_{C_1\cap C_2}$ is the opposite of the restriction map, and other maps are canonical restriction maps. This sequence gives rise to an exact sequence of \'etale sheaves on $C\colon$
		$$	0\to\GG_{m,Z^f_L}\to i_{1*}\GG_{m,C_1}\oplus i_{2*}\GG_{m,C_2}\to i_{3*}\GG_{m,C_1\cap C_2}\to 0.$$
		By the long exact sequence of \'etale cohomology, we have the following exact sequence:
		$$ H^1_{\et}(Z^f_L,i_{3*}\GG_{m,C_1\cap C_2})\to 
		H^2_{\et}(Z^f_L,\GG_{m})\to H^2_{\et}(Z^f_L,i_{1*}\GG_{m,C_1}\oplus i_{2*}\GG_{m,C_2})\to H^2_{\et}(Z^f_L,i_{3*}\GG_{m,C_1\cap C_2}).$$
		For $i_1,$ $i_2$ and $i_3$ are closed embeddings, it gives the following exact sequence:
		\begin{equation}\label{ee1}
			H^1_{\et}(C_1\cap C_2,\GG_{m})\to 
			\Br(Z^f_L)\to \Br(C_1)\oplus \Br(C_2)\to \Br(C_1\cap C_2).
		\end{equation}
		By our choice, two different places $v_{p_1}$ and $v_{p_2}$ split in $L,$ so we have number fields $L(\sqrt{p_1}),$ $L(\sqrt{p_2}),$ $L(\sqrt{p_1p_2}),$ denoted by $L_{10},L_{20},L_{30}$ respectively. And 
		$$C_1\cong (\Spec L_{10}\times_{\Spec L} \PP^1)\bigsqcup (\Spec L_{20}\times_{\Spec L} \PP^1)\bigsqcup
		(\Spec L_{30}\times_{\Spec L}\PP^1).$$ 
		So $\Br(C_1)\cong \bigoplus_{i=1}^3\Br(L_{i0}).$
		
		Similarly, we have number fields $L(\sqrt{p_3}),$ $L(\sqrt[3]{p_3}),$ denoted by $L_{01},L_{02}$ respectively. And 
		$$C_2\cong (\PP^1\times_{\Spec L} \Spec L_{01}) \bigsqcup (\PP^1\times_{\Spec L} \Spec L_{02}).$$
		Then $\Br(C_2)\cong \bigoplus_{j=1}^2 \Br(L_{0j}).$
		
		Since the different places $v_{p_1},$ $v_{p_2}$ and $v_{p_3}$ split in $L,$ for any $i\in \{1,2,3\},$ and any $j\in \{1,2\},$ we have number fields
		$L_{i0}\otimes_L  L_{0j},$ denoted by $L_{ij}.$ Then
		% $L(\sqrt{p_1},\sqrt{p_3}), L(\sqrt{p_2},\sqrt{p_3}),\\ L(\sqrt{p_1p_2},\sqrt{p_3}),	 L(\sqrt{p_1},\sqrt[3]{p_3}), L(\sqrt{p_2},\sqrt[3]{p_3}), L(\sqrt{p_1p_2},\sqrt[3]{p_3}),$ denoted by 	 $L_{11},L_{21},L_{31},L_{12},L_{22},L_{32}$ respectively.
		$$C_1\cap C_2\cong \bigsqcup_{i=1}^3\bigsqcup_{j=1}^2 \Spec L_{ij}.$$
		%\bigsqcup L(\sqrt{p_1p_2},\sqrt{p_3})\bigsqcup L(\sqrt{p_1},\sqrt[3]{p_3})\bigsqcup L(\sqrt{p_2},\sqrt[3]{p_3})\bigsqcup L(\sqrt{p_1p_2},\sqrt[3]{p_3}).$$
		So $\Br(C_1\cap C_2)\cong  \bigoplus_{i=1}^3  \bigoplus_{j=1}^2 \Br(L_{ij}).$
		%\oplus \Br(L(\sqrt{p_2},\sqrt{p_3}))\oplus \Br(L(\sqrt{p_1p_2},\sqrt{p_3}))\oplus \Br(L(\sqrt{p_1},\sqrt[3]{p_3}))\oplus \Br(L(\sqrt{p_2},\sqrt[3]{p_3}))\oplus \Br(L(\sqrt{p_1p_2},\sqrt[3]{p_3})).$
		
		By Hilbert's Theorem $90,$ we have $H^1_{\et}(C_1\cap C_2,\GG_{m})=0.$
		By the exact sequence (\ref{ee1}), we have an exact sequence: 
		\begin{equation}\label{ee2}
			0\to \Br(Z^f_L) \to \Br(C_1)\oplus \Br(C_2) \to \Br(C_1\cap C_2).
		\end{equation}
		By Lemma \ref{lemma zero dimension scheme violate HP}, the set $Z^f_L(\AA_L)\neq \emptyset.$ Indeed, from the prove of Lemma \ref{lemma zero dimension scheme violate HP}, the set $C_1(\AA_L)\neq \emptyset.$ We take an adelic point $(P_{v'})_{v'\in \Omega_L}\in C_1(\AA_L),$ then the evaluation of elements in $\Br(C_1)$ on this adelic point gives a map: $\Br(C_1)\to \bigoplus_{v'\in \Omega_L}\Br(L_{v'}),$ 
		which makes the following diagram: 
		$$ \xymatrix{
			\Br(L)\ar[rd] \ar[r]&\Br(Z^f_L) \ar[r]& Br(C_1) \ar[ld] \\
			&\bigoplus_{v'\in \Omega_L}\Br(L_{v'})&  \\
		}$$
		commutative. By the reciprocity law of global class field theory, the map $\Br(L)\to \bigoplus_{v'\in \Omega_L}\Br(L_{v'})$ is injective, so the natural map $\Br(L)\to \Br(Z^f_L)$ is injective. We have the following commutative diagram:
		$$ \xymatrix{
			& \Br(L)\ar@{^(->}[d] & & &\\
			0  \ar[r]&\Br(Z^f_L) \ar[r]\ar@{=}[d]& \Br(C_1)\oplus \Br(C_2) \ar[r]\ar[d]^{\cong}& \Br(C_1\cap C_2)\ar[d]^{\cong}& \\
			0\ar[r]&\Br(Z^f_L)\ar[r] & \bigoplus_{i=1}^3\Br(L_{i0}) \oplus \bigoplus_{j=1}^2 \Br(L_{0j}) 
			\ar[r]& \bigoplus_{i=1}^3 \bigoplus_{j=1}^2 \Br(L_{ij})& \\
		}$$ 
		with exact rows. Next, we will prove that the natural map $\Br(L)\to \Br(Z^f_L)$ is surjective. By the commutative diagram, we need to prove that the sequence:  
		$$\Br(L)\to \bigoplus_{i=1}^3\Br(L_{i0}) \oplus \bigoplus_{j=1}^2 \Br(L_{0j}) 
		\to \bigoplus_{i=1}^3 \bigoplus_{j=1}^2 \Br(L_{ij})$$
		is exact. Notice that by our choice, the map $\bigoplus_{j=1}^2 \Br(L_{0j}) 
		\to \bigoplus_{i=1}^3 \bigoplus_{j=1}^2 \Br(L_{ij})$ is the opposite of the restriction map, and other maps are canonical restriction maps. Take an element $(\alpha_{i0},\alpha_{0j})\in \bigoplus_{i=1}^3\Br(L_{i0}) \oplus \bigoplus_{j=1}^2 \Br(L_{0j}).$ Suppose that it goes to zero in $\bigoplus_{i=1}^3 \bigoplus_{j=1}^2 \Br(L_{ij}).$ So the restrictions of $\alpha_{i0}$ and $\alpha_{0j}$ to $\Br(L_{ij})$ coincide. Also consider the adelic point $(P_{v'})_{v'\in \Omega_L}$ and the map: $\Br(C_1)\to \bigoplus_{v'\in \Omega_L}\Br(L_{v'}).$ By $\Br(C_1)\cong \bigoplus_{i=1}^3\Br(L_{i0}),$ we view $(\alpha_{i0})$ as an element in $\Br(C_1)$ and let  $(a_{v'})_{v'\in \Omega_L}$ be its image in $\bigoplus_{v'\in \Omega_L}\Br(L_{v'}).$ Then for any $j\in \{1,2\},$ we have the following commutative diagram:
		$$ $$
		\begin{equation}\label{ee3}
			\xymatrix{
				& & \Br(C_1) \ar[d]& & \\
				0\ar[r]&\Br(L) \ar[d]\ar[r] & \bigoplus_{v'\in \Omega_L}\Br(L_{v'}) \ar[d]\ar[r]& \QQ/\ZZ \ar[d]^{[L_{0j}:L]}\ar[r]& 0 \\
				0\ar[r]&\Br(L_{0j}) \ar[r] & \bigoplus_{v'\in \Omega_L}\Br(L_{0j}\otimes_L L_{v'}) \ar[r]& \QQ/\ZZ \ar[r]& 0. \\
			}
		\end{equation}
		By the reciprocity law of global class field theory, two rows of this diagram are exact. For the restrictions of $\alpha_{i0}$ and $\alpha_{0j}$ to $\Br(L_{ij})$ coincide, the restrictions of  $(a_{v'})_{v'\in \Omega_L}$ and $\alpha_{0j}$ to $\bigoplus_{v'\in \Omega_L}\Br(L_{0j}\otimes_L L_{v'})$ coincide. So $[L_{0j}:L]\sum_{v'\in \Omega_L}\inv_{v'}(a_{v'})=0$ in $\QQ/\ZZ.$ For the degrees $[L_{01}:L]=2$ and $[L_{02}:L]=3,$ we have $\sum_{v'\in \Omega_L}\inv_{v'}(a_{v'})=0$ in $\QQ/\ZZ.$ By the exact sequence of the first row, let $a\in \Br(L)$ be the element such 
		that its image in $\bigoplus_{v'\in \Omega_L}\Br(L_{v'})$ equals $(a_{v'})_{v'\in \Omega_L}.$ Let $a|_{L_{0j}}$ and $a|_{L_{i0}}$ be the restrictions of $a$ to $\Br(L_{0j})$ and $\Br(L_{i0})$ respectively. Then from the diagram (\ref{ee3}), we have $a|_{L_{0j}}=\alpha_{0j}.$ For any $i\in \{1,2,3\},$ we consider the element $\alpha_{i0}-a|_{L_{i0}}.$ For the restrictions of  $\alpha_{i0}-a|_{L_{i0}}$ and $\alpha_{0j}-a|_{L_{0j}}=0$ to $\Br(L_{ij})$ coincide, they are zero in $\Br(L_{ij}).$ By the standard restriction-corestriction argument, we have $[L_{0j}:L](\alpha_{i0}-a|_{L_{i0}})=0$ in $\Br(L_{i0}).$ For the degrees $[L_{01}:L]=2$ and $[L_{02}:L]=3,$ we have  $a|_{L_{i0}}=\alpha_{i0}.$ So the element $a$ maps to the element $(\alpha_{i0},\alpha_{0j}),$ which implies that the map $\Br(L)\to \Br(Z^f_L)$ is surjective.
	\end{proof}
	
	\begin{remark}
		In our proof, the map $\Br(C_1)\to \bigoplus_{v'\in \Omega_L}\Br(L_{v'}),$ depends on the choice of the adelic point $(P_{v'})_{v'\in \Omega_L}$ in $C_1(\AA_L).$ We use this adelic point to illustrate that  the map $\Br(L)\to \Br(Z^f_L)$ is injective. In order to prove this injection, by
		using the information from $C_2,$ the curve $Z^f_L$ contains closed points of degree $2$ and $3,$ then one can use the standard restriction-corestriction argument to get this injection. The idea to proof that this map is surjective, comes from \cite[Proposition 3.1]{HS14}.
	\end{remark}
	
	\subsubsection{Choosing another curve with respect to an extension}\label{subsection choose a polynomial with respect an extension2}
	In this subsubsection, we will choose another curve with some given arithmetic properties.
	Given an extension of number fields $L/K,$ 
	similar to the choice of $p_1,$ we can choose an odd prime element $p_4\in \Ocal_K$ satisfying the following conditions:
	\begin{itemize}
		%\item $\tau_v(a)<0$ for all $v\in S\cap \infty_K,$
		\item $\tau_v(p_4)>0$ for all $v\in \infty_K^r,$
		\item $p_4\in K_v^{\times 2}$  for all $v\in 2_K,$ 
		\item $p_4$ splits in $L,$
		\item $v_{p_4}\notin\{v_{p_1},v_{p_2},v_{p_3}\}.$ 
	\end{itemize}
	By Lemmas \ref{lemma for hilbert prime element -1} and \ref{lemma dirichlet}, we choose an odd prime element $p_5\in \Ocal_K$ satisfying the following conditions:
	\begin{itemize}
		%\item $\tau_v(a)<0$ for all $v\in S\cap \infty_K,$
		\item $(p_4,p_5)_{v_{p_4}}=-1,$
		\item $v_{p_5}\notin\{v_{p_1},v_{p_2},v_{p_3},v_{p_4}\}.$ 
	\end{itemize}
	Similarly, by Lemmas \ref{lemma for hilbert prime element -1} and \ref{lemma dirichlet}, we choose an odd prime element $p_6\in \Ocal_K$ satisfying the following conditions:
	\begin{itemize}
		%\item $\tau_v(a)<0$ for all $v\in S\cap \infty_K,$
		\item $(p_4,p_6)_{v_{p_4}}=-1,$
		\item $(p_5,p_6)_{v_{p_5}}=1,$
		\item $v_{p_6}\notin\{v_{p_1},v_{p_2},v_{p_3},v_{p_4},v_{p_5}\}.$ 
	\end{itemize}
	Let $g(x_0,x_1;y_0,y_1)=(x_0^2-p_4x_1^2)(x_0^2-p_5x_1^2)(x_0^2-p_4p_5x_1^2)(y_0^2-p_6y_1^2)(y_0^3-p_4y_1^3)$ be a bi-homogeneous polynomial, and let $Z^g$ be the zero locus of $g$ in $\PP^1\times\PP^1$ with bi-homogeneous coordinates $(x_0:x_1)\times(y_0:y_1).$ With the notation, we have the following lemma.
	
	\begin{lemma}\label{lemma zero dimension scheme without local point}
		Let $Z^g\subset \PP^1\times\PP^1$ be the zero locus defined over $K$ by the bi-homogeneous polynomial $g(x_0,x_1;y_0,y_1).$  Then $Z^g(\AA_K^{\{v_{p_4}\}})\neq\emptyset$ but $Z^g(K_{v_{p_4}})=\emptyset.$
	\end{lemma}
	
	\begin{proof}
		Suppose that $v\in \infty_K\cup 2_K.$  Then, by the choice of $p_4,$ we have $p_4\in K_v^{\times 2}.$ So the equation $x_0^2-p_4x_1^2=0$ has a $K_v$-solution in $\PP^1$ with homogeneous coordinates $(x_0:x_1).$\\
		Suppose that $v=v_{p_5}.$ Then, by the choice of $p_6,$ we have $(p_5,p_6)_{v_{p_5}}=1.$ By Lemma \ref{lemma for two prime elements one is square for another}, we have $p_6\in K_{v_{p_5}}^{\times 2}.$
		So the equation $y_0^2-p_6y_1^2=0$ has a $K_v$-solution in $\PP^1$ with homogeneous coordinates $(y_0:y_1).$\\
		Suppose that $v\in \Omega_K\backslash( \infty_K\cup 2_K\cup\{v_{p_4},v_{p_5}\}),$ then,
		by the quadratic reciprocity law, at least one of equations: $x_0^2-p_4x_1^2=0,~x_0^2-p_5x_1^2=0,~x_0^2-p_4p_5x_1^2=0,$ has a $K_v$-solution in $\PP^1$ with homogeneous coordinates $(x_0:x_1).$\\
		So  $Z^g(\AA_K^{\{v_{p_4}\}})\neq\emptyset.$
		
		Suppose that $v=v_{p_4}.$ Then the equations $x_0^2-p_4x_1^2=0,~x_0^2-p_4p_5x_1^2=0$ and $y_0^3-p_4y_1^3=0$ has no $K_v$-solution in $\PP^1$ with homogeneous coordinates $(x_0:x_1)$ and $(y_0:y_1)$ respectively.  By the choice of $p_5,p_6,$ we have $(p_4,p_5)_{v_{p_4}}=-1$ and $(p_4,p_6)_{v_{p_4}}=-1.$ By Lemma \ref{lemma for two prime elements one is square for another}, we have $p_5\notin K_{v_{p_4}}^{\times 2}$ and $p_6\notin K_{v_{p_4}}^{\times 2}.$ So the equations $x_0^2-p_5x_1^2=0$ and $y_0^2-p_6y_1^2=0$ have no $K_v$-solution in $\PP^1$ with homogeneous coordinates $(x_0:x_1)$ and $(y_0:y_1)$ respectively. So $Z^g(K_{v_{p_4}})=\emptyset.$
	\end{proof}

	\begin{eg}\label{example: choose prime elements}
		For $K=\QQ$ and $L=\QQ(i),$ let prime elements $(p_1,p_2,p_3,p_4,p_5,p_6)=(17,13,53,41,3,13).$ Then they satisfy all chosen conditions of Subsubsections \ref{subsection choose a polynomial with respect an extension} and \ref{subsection choose a polynomial with respect an extension2}. They will be used for construction of our explicit unconditional example.
	\end{eg}

	\begin{theorem}\label{theorem main result: non-invariance of the Hasse principle with BMO for odd degree}
		For any nontrivial extension of number fields $L/K,$ assuming that Conjecture \ref{conjecture Stoll} holds over $K,$ there exists a smooth, projective, and geometrically connected surface  $X$ defined over $K$ such that
		\begin{itemize} 
			\item the surface $X$ is a counterexample to the Hasse principle, and its failure of the Hasse principle is explained by the Brauer-Manin obstruction,
			\item the surface $X_L$ is a counterexample to the Hasse principle, but its failure of the Hasse principle cannot be explained by the Brauer-Manin obstruction.
		\end{itemize}
	\end{theorem}

	\begin{proof}
		We will construct a smooth, projective, and geometrically connected surface $X.$
		For the extension $L/K,$ we choose odd prime elements $p_1,p_2,p_3,p_4,p_5,p_6\in \Ocal_K$ 
		as in Subsubsections \ref{subsection choose a polynomial with respect an extension} and \ref{subsection choose a polynomial with respect an extension2}. Let $f(x_0,x_1;y_0,y_1)=(x_0^2-p_1x_1^2)(x_0^2-p_2x_1^2)(x_0^2-p_1p_2x_1^2)(y_0^2-p_3y_1^2)(y_0^3-p_3y_1^3)$ and $g(x_0,x_1;y_0,y_1)=(x_0^2-p_4x_1^2)(x_0^2-p_5x_1^2)(x_0^2-p_4p_5x_1^2)(y_0^2-p_6y_1^2)(y_0^3-p_4y_1^3)$ be two bi-homogeneous polynomials, and let $Z^f$ and $Z^g$ be the zero loci of $f$ and $g$ respectively in $\PP^1\times\PP^1$ with bi-homogeneous coordinates $(x_0:x_1)\times(y_0:y_1).$ Let $(u_0:u_1)\times(x_0:x_1)\times(y_0:y_1)$ be the coordinates of $\PP^1\times\PP^1\times\PP^1,$ and let $s'=u_0g(x_0,x_1;y_0,y_1)+u_1f(x_0,x_1;y_0,y_1)\in \Gamma(\PP^1\times\PP^1\times\PP^1,\Ocal(1,6,5)).$ 
		Let $X'$ be the locus defined by $s'=0$ in $\PP^1\times\PP^1\times\PP^1.$ For the curves $Z^f$ and $Z^g$ meet transversally, the locus $X'$ is smooth. Let $R$ be the locus over which the composition $ X'\hookrightarrow \PP^1\times\PP^1\times\PP^1  \stackrel{pr_1}\to\PP^1$ is not smooth. Then by \cite[Chapter III. Corollary 10.7]{Ha97}, it is finite over $K.$ 
		By the assumption that Conjecture \ref{conjecture Stoll} holds over $K,$ and Lemma \ref{lemma stoll conjecture}, we can take a smooth, projective, and geometrically connected curve $C$ defined over $K$ such that the triple $(C,K,L)$ is of type $I.$
		By Lemma \ref{lemma choose base change morphism}, we can choose a $K$-morphism $\gamma\colon  C\to \PP^1$ such that $\gamma(C(L)\backslash C(K))=\{0\}\subset \PP^1(K),$ $\gamma(C(K))=\{\infty\}\subset \PP^1(K),$ and that $\gamma$ is \'etale over $R.$ 
		Let $B=C\times \PP^1\times \PP^1,$ and let $(\gamma,id)\colon B\to\PP^1\times \PP^1\times \PP^1.$ Let $\Lcal=(\gamma,id)^*\Ocal(1,6,5),$ and let $s=(\gamma,id)^* (s')\in \Gamma(B,\Lcal).$ Let $X$ be the zero locus of $s$ in $B.$ 
		By the same argument as in the proof of Theorem \ref{theorem main result: non-invariance of weak approximation with BMO}, the surface $X$ is smooth, projective, and geometrically connected.
		Let $\beta\colon X \hookrightarrow B=C\times \PP^1\times \PP^1 \stackrel{pr_1}\to C$ be the composition morphism. By our construction, we have the following Cartesian diagram:
		$$\xymatrix{
			X \ar@{^(->}[d]\ar[r] \ar@/_3.5pc/[dd]_{\beta} & X'  \ar@{^(->}[d]  \\
			C\times \PP^1\times \PP^1 \ar^{pr_1}[d] \ar^{(\gamma,id)}[r]&  \PP^1\times \PP^1\times \PP^1 \ar^{pr_1}[d]\\
			C \ar^{\gamma}[r] & \PP^1  
		}$$
		
		Next, we will check that the surface $X$ has the properties.	
		
		We will show $X(\AA_K)\neq \emptyset.$
		For any $P\in C(K),$ the fiber  $\beta^{-1}(P)\cong Z^g.$ By Lemma \ref{lemma zero dimension scheme without local point}, the set $Z^g(\AA_K^{\{v_{p_4}\}})\neq \emptyset.$ So the set $X(\AA_K^{\{v_{p_4}\}})\neq \emptyset.$ For $v_{p_4}$ splits in $L,$ we take a place $v'\in \Omega_L^f$ above $v_{p_4}$ such that $K_{v_{p_4}}=L_{v'}.$ By Lemma \ref{lemma zero dimension scheme violate HP}, the set $Z^f(\AA_L)\neq \emptyset.$ Take a point $Q\in C(L)\backslash C(K),$ then the fiber $\beta^{-1}(Q)\cong Z^f_L.$	
		%$\beta^{-1}(C(L))=[V_\infty\times C(K)]\bigcup [V_{0L}\times (C(L)\backslash C(K))]\subset X_L,$ 
		We have  $X(K_{v_{p_4}})=X_L(L_{v'})\supset \beta^{-1}(Q)(L_{v'})\cong Z^f(L_{v'})\neq \emptyset.$ So the set $X(\AA_K)\neq \emptyset.$\\
		We will show $X(\AA_K)^{\Br}=\emptyset.$	 By our choice and Conjecture \ref{conjecture Stoll}, the set $C(K)$ is finite, and $C(K)=pr^{\infty_K}(C(\AA_K)^{\Br}).$ By the functoriality of Brauer-Manin
		pairing, we have  $pr^{\infty_K}(X(\AA_K)^{\Br})\subset  \bigsqcup_{P\in C(K)}\beta^{-1}(P)(\AA_K^{\infty_K}).$ But by Lemma \ref{lemma zero dimension scheme without local point}, the set $Z^g(K_{v_{p_4}})=\emptyset,$ so we have
		$pr^{\infty_K}(X(\AA_K)^{\Br})\subset \bigsqcup_{P\in C(K)}\beta^{-1}(P)(\AA_K^{\infty_K})\cong Z^g(\AA_K^{\infty_K})\times C(K)=\emptyset,$ which implies that $ X(\AA_K)^{\Br}=\emptyset.$ \\
		So, the surface $X$ is a counterexample to the Hasse principle, and its failure of the Hasse principle is explained by the Brauer-Manin obstruction.
		
		We will show $X_L(\AA_L)^{\Br}\neq \emptyset.$ Take a point $Q\in C(L)\backslash C(K).$
		By Lemma \ref{lemma zero dimension scheme trival Brauer group pairing}, the set $Z^f_L(\AA_L)^\Br=Z^f_L(\AA_L).$ By Lemma \ref{lemma zero dimension scheme violate HP}, it is nonempty.
		By the functoriality of Brauer-Manin pairing, the set  $X_L(\AA_L)^{\Br}$ contains
		$\beta^{-1}(Q)(\AA_L)^{\Br}\cong Z^f_L(\AA_L)^{\Br},$ so $X_L(\AA_L)^{\Br}\neq \emptyset.$\\
		We will show $X(L)=\emptyset.$ 
		By Lemma \ref{lemma zero dimension scheme without local point} and the condition that $v_{p_4}$ splits in $L,$ we have $Z^g(\AA_L)=\emptyset,$ so the set $Z^g(L)=\emptyset.$ By Lemma \ref{lemma zero dimension scheme violate HP}, the set $Z^f(L)=\emptyset.$ 
		Since each $L$-rational fiber of $\beta$ is isomorphic to $Z^g_L$ or $Z^f_L,$ the set $X(L)=\emptyset.$	\\
		So,
		%	$\subset(\beta^{-1}(C(L)))(L)=([V_\infty\times C(K)])(L)\bigcup ([V_0\times (C(L)\backslash C(K))])(L),$ the set $X(L)=\emptyset.$ 
		the variety $X_L$ is a counterexample to the Hasse principle, but its failure of the Hasse principle cannot be explained by the Brauer-Manin obstruction.
	\end{proof}

	\section{Explicit unconditional examples}

	In this section, let $K=\QQ$ and $L=\QQ(i).$ For this extension $L/K,$ we will give explicit examples without assuming Conjecture \ref{conjecture Stoll} for Theorem \ref{theorem main result: non-invariance of weak approximation with BMO} and Theorem \ref{theorem main result: non-invariance of the Hasse principle with BMO for odd degree}.

	\subsection{Choosing an elliptic curve and a dominant morphism} 
	For the extension $L/K,$ as in the proof of  Theorem \ref{theorem main result: non-invariance of weak approximation with BMO} and Theorem \ref{theorem main result: non-invariance of the Hasse principle with BMO for odd degree},  we can choose a common elliptic curve over $K$ for these examples.
	\subsubsection{Choosing an elliptic curve} For the extension $L/K,$ we will choose an elliptic curve such that the triple $(E,K,L)$ is of  type $I.$ Let $E$ be an elliptic curve defined over $\QQ$ by a homogeneous equation:
	$$w_1^2w_2=w_0^3-16w_2^3$$ in $\PP^2$ with homogeneous coordinates $(w_0:w_1:w_2).$ Its quadratic twist $E^{(-1)}$ is isomorphic to an elliptic curve defined by a homogeneous equation:
	$w_1^2w_2=w_0^3+16w_2^3.$ The elliptic curves $E$ and $E^{(-1)}$ over $\QQ,$ are  of analytic rank $0.$ Then the Tate-Shafarevich group $\Sha(E,K)$ is finite, so the curve $E$ satisfies weak approximation with Brauer-Manin obstruction off $\infty_K.$ The Mordell-Weil groups $E(K)$ and $E^{(-1)}(K)$ are both finite, so the group $E(L)$ is finite. Using \cite[SageMath]{St12}, we check that the Mordell-Weil group $E(K)=\{(0:1:0)\}$ and $E(L)=\{(0:\pm 4i:1),(0:1:0)\}.$ So the triple $(E,K,L)$ is of type $I.$
	
	%And $E$ and $E^{(-1)}$ are elliptic curves over $\QQ$ of analytic rank $0.$ Then the Tate-Shafarevich group $\Sha(E,\QQ)$ and $\Sha(E^{(-1)},\QQ)$ are both finite. The curves $E_K$ and $E_L$ satisfy weak approximation with Brauer-Manin obstruction off $\infty_K$ and $\infty_L$ respectively. The Mordell-Weil group $E(K)$ and $E(L)$ are both finite. Indeed, the Mordell-Weil group $E(K)=\{(0:1:0)\}$ and $E(L)=\{(0:\pm 4i:1),(0:1:0)\}.$ 

	\subsubsection{Choosing a dominant morphism}
	We choose the following dominant morphism from the elliptic curve $E$ to $\PP^1,$ which satisfies some conditions of Lemma \ref{lemma choose base change morphism}.

	Let $\PP^2\backslash\{(1:0:0),(-16:0:1),(-1:\pm \sqrt{15}i:1)\}\to \PP^1$ be a morphism over $\QQ$ given by $(w_0:w_1:w_2)\mapsto(w_0w_2+w_1^2+16w_2^2:w_0w_1+w_1w_2).$ Composing  the natural inclusion $E\hookrightarrow \PP^2\backslash\{(1:0:0),(-16:0:1),(-1:\pm \sqrt{15}i:1)\}$ with it, we get a morphism $\gamma\colon E\to \PP^1,$ which is a dominant morphism of degree $6.$ The dominant morphism $\gamma$ maps $E(K)$ to $\{\infty\}=\{(1:0)\},$ and maps $(0:\pm 4i:1)$ to $0:=(0:1).$  By B\'ezout's Theorem \cite[Chapter I. Corollary 7.8]{Ha97} and calculation, the branch locus of $\gamma$  is contained in $\PP^1\backslash\{\infty\}.$ Let $(u_0:1)\in \PP^1$ be a branch point of $\gamma.$ For fixed $u_0,$ we use Jacobian criterion for the intersection of two curves $E$ and $w_0w_2+w_1^2+16w_2^2=(w_0w_1+w_1w_2)u_0$ in $\PP^2.$ For the point $(0:1:0)\in \PP^2$ is not in this intersection, we 
	let $w_2=1$ to dehomogenize these two curves. By Jacobian criterion, the branch locus satisfies the following equations:
	\begin{equation*}
		\begin{cases}
			w_1^2=w_0^3-16\\
			w_1^2+w_0+16=w_1(w_0+1)u_0\\
			3(2w_1-w_0u_0-u_0)w_0^2+2w_1(1-w_1u_0)=0.
		\end{cases}
	\end{equation*}
	Then %we have $\{w_0|w_0^6 + 3w_0^5 - 3w_0^4 - 65w_0^3 - 96w_0^2 - 32=0\}.$ By $w_1^2=w_0^3-16,$ we have $\{w_1|w_1^{12} - 45w_1^{10} - 99w_1^8 + 66879w_1^6 - 723600w_1^4 - 55302912w_1^2 - 483176448=0\}.$ And $u_0=\frac{(w_0^3+w_0)}{(w_0+1)w_1}.$ Let $u_0$ be a branch point, then the degree $[\QQ(u_0):\QQ]=12.$ By Hurwitz's Theorem \cite[Chapter IV. Corollary 2.4]{Ha97}, 
	the branch locus equals 
	$$\left\{(u_0:1)\big{|} u_0^{12} + \frac{60627u_0^{10}}{4913} + \frac{159828u_0^8}{4913} - \frac{3505917u_0^6}{19652} - \frac{42057961u_0^4}{58956} + \frac{76076u_0^2}{14739} - \frac{4112}{132651}=0\right\}.$$ Let $(u_0:1)$ be a branch point, then the degree $[\QQ(u_0):\QQ]=12.$
	% consists of 12-points. So the points of this branch locus conjugate to each other over $\QQ.$

	\subsection{An explicit unconditional example for Theorem \ref{theorem main result: non-invariance of weak approximation with BMO}}\label{subsection main example1}
	
	For $K=\QQ$ and $L=\QQ(i),$ in this subsection,
	we will construct a smooth, projective, and geometrically connected surface having properties of Theorem \ref{theorem main result: non-invariance of weak approximation with BMO}.

	\subsubsection{Construction of a smooth, projective, and geometrically connected surface}
	We will construct a smooth, projective, and geometrically connected surface $X$ as in Theorem \ref{theorem main result: non-invariance of weak approximation with BMO}. 
	Let $(u_0:u_1)\times(x_0:x_1:x_2)$ be the coordinates of $\PP^1\times\PP^2,$ and let $s'=u_0(x_0^2+x_1^2-x_2^2)+u_1(x_0^2-x_1^2)\in \Gamma(\PP^1\times\PP^2,\Ocal(1,2)).$ The locus $X'$ defined by $s'=0$ in $\PP^1\times\PP^2$ is smooth. 
	Let $R$ be the locus over which the composition $ X'\hookrightarrow \PP^1\times\PP^2  \stackrel{pr_1}\to\PP^1$ is not smooth.  By calculation, the locus 
	$R=\{(0:1),(\pm 1:1)\}.$ 
	Let $B=E\times \PP^2,$ and let $(\gamma,id)\colon B\to\PP^1\times \PP^2.$ Let $\Lcal=(\gamma,id)^*\Ocal(1,2),$ and let $s=(\gamma,id)^* (s')\in \Gamma(B,\Lcal).$ Let $X$ be the zero locus of $s$ in $B.$
	For the locus $R$
	does not intersect with the branch locus of  $\gamma\colon E\to \PP^1,$ the surface $X$ is smooth.
	So it is smooth, projective, and geometrically connected. 
	By our construction, the surface $X$ is defined by the following equations: 
	\begin{equation*}
		\begin{cases}
			(w_0w_2+w_1^2+16w_2^2)(x_0^2+x_1^2-x_2^2)+(w_0w_1+w_1w_2)(x_0^2-x_1^2)=0\\ 
			w_1^2w_2=w_0^3-16w_2^3
		\end{cases}
	\end{equation*}
	in $\PP^2\times \PP^2$ with bi-homogeneous coordinates $(w_0:w_1:w_2)\times(x_0:x_1:x_2).$
	For this surface $X,$ we have the following proposition.

	\begin{prop}\label{example1: main result: non-invariance of weak approximation with BMO}
		For $K=\QQ$ and $L=\QQ(i),$  the smooth, projective, and geometrically connected surface $X$ has the following properties.	
		\begin{itemize}
			\item The surface $X$ has a $K$-rational point, and satisfies weak approximation with Brauer-Manin obstruction %, \'etale Brauer-Manin obstruction 
			off $\infty_K.$
			\item The surface $X_L$ does not satisfy weak approximation with Brauer-Manin obstruction %, \'etale Brauer-Manin obstruction
			off $T$ for any finite subset $T\subset \Omega_L.$ 
		\end{itemize}
	\end{prop}

	\begin{proof}
		This is the same as in the proof of Theorem \ref{theorem main result: non-invariance of weak approximation with BMO}.
	\end{proof}

	%
	%\begin{proof}
	%	By the same argument as in the proof of Theorem \ref{theorem main result: non-invariance of the Hasse principle with BMO for odd degree}, the variety $X$ is smooth, projective, and geometrically connected. 
	%	For the surface $V_\infty$ has a $K$-rational point $(x,y,z)=(0,2,0),$ the set $X(K)\neq \emptyset.$
	%	For $\Br(V_\infty)/\Br(K)=0,$ according to \cite[Theorem B]{CTSSD87a,CTSSD87b}, the surface $V_\infty$ satisfies weak approximation.  Because the curve $E$ satisfies Conjecture \ref{conjecture Stoll}, using Lemma \ref{lemma fiber criterion for wabm} for the morphism $\beta,$ the variety $X$ has a $K$-rational point, and satisfies weak approximation with Brauer-Manin obstruction, \'etale Brauer-Manin obstruction off $\infty_K.$ 
	%	
	%	By Proposition \ref{Poonen's main proposition}, we have $\alpha_L^*\colon\Br(B_L)\to \Br(X_L)$ is an isomorphism. By construction, we have $\beta^{-1}(E(L))=V_{\infty L}\bigcup V_{0 L}\bigcup V_{0 L}.$ 
	%	By Example \ref{example1: construction of V_0} and Proposition \ref{proposition the valuation of Brauer group on local points are fixed outside S and take two value on S}, the surface $V_{0L}$ does not satisfy weak approximation off $\infty_L$ and the set $V_0(L)\neq\emptyset.$ By Lemma \ref{lemma fiber criterion for not wabm}, the variety $X$ does not satisfy weak approximation with Brauer-Manin obstruction, \'etale Brauer-Manin obstruction off $\infty_L.$ 
	%\end{proof}

	\subsection{An explicit unconditional example for Theorem \ref{theorem main result: non-invariance of the Hasse principle with BMO for odd degree}}\label{subsection main example2}
	
	For $K=\QQ$ and $L=\QQ(i),$ in this subsection,
	we will construct a smooth, projective, and geometrically connected surface having properties of Theorem \ref{theorem main result: non-invariance of the Hasse principle with BMO for odd degree}.

	\subsubsection{Construction of a smooth, projective, and geometrically connected surface} We choose odd prime elements $(p_1,p_2,p_3,p_4,p_5,p_6)=(17,13,53,41,3,13)$ as in Example \ref{example: choose prime elements}. Then they satisfies all chosen conditions of Subsubsections \ref{subsection choose a polynomial with respect an extension} and \ref{subsection choose a polynomial with respect an extension2}. Let $f(x_0,x_1;y_0,y_1)=(x_0^2-17x_1^2)(x_0^2-13x_1^2)(x_0^2-221x_1^2)(y_0^2-53y_1^2)(y_0^3-53y_1^3)$ and  $g(x_0,x_1;y_0,y_1)=(x_0^2-41x_1^2)(x_0^2-3x_1^2)(x_0^2-123x_1^2)(y_0^2-13y_1^2)(y_0^3-41y_1^3)$ be two bi-homogeneous polynomials. Let $Z^f$ and $Z^g$ be the zero loci of $f$ and $g$ respectively in $\PP^1\times\PP^1$ with bi-homogeneous coordinates $(x_0:x_1)\times(y_0:y_1).$ Let $(u_0:u_1)\times(x_0:x_1)\times(y_0:y_1)$ be the coordinates of $\PP^1\times\PP^1\times\PP^1,$ and let $s'=u_0g(x_0,x_1;y_0,y_1)+u_1f(x_0,x_1;y_0,y_1)\in \Gamma(\PP^1\times\PP^1\times\PP^1,\Ocal(1,6,5)).$ 
	The locus $X'$ defined by $s'=0$ in $\PP^1\times\PP^1\times\PP^1$ is smooth. Let $R$ be the locus over which the composition $ X'\hookrightarrow \PP^1\times\PP^1\times\PP^1  \stackrel{pr_1}\to\PP^1$ is not smooth. It is finite over $\QQ.$ We can use computer to calculate this locus, and we give the calculation in Appendix \ref{appendix branch locus}.
	Let $B=E\times \PP^1\times \PP^1,$ and let $(\gamma,id)\colon B\to\PP^1\times \PP^1\times \PP^1.$ Let $\Lcal=(\gamma,id)^*\Ocal(1,6,5),$ and let $s=(\gamma,id)^* (s')\in \Gamma(B,\Lcal).$ Let $X$ be the zero locus of $s$ in $B.$ For the locus $R$
	does not intersect with the branch locus of  $\gamma\colon E\to \PP^1,$ the surface $X$ is smooth.
	So it is smooth, projective, and geometrically connected. 
	By our construction, the surface $X$ is defined by the following two equations: 
	\begin{equation*}
		\begin{cases}
			(w_0w_2+w_1^2+16w_2^2)(x_0^2-41x_1^2)(x_0^2-3x_1^2)(x_0^2-123x_1^2)(y_0^2-13y_1^2)(y_0^3-41y_1^3)\\
			+(w_0w_1+w_1w_2)(x_0^2-17x_1^2)(x_0^2-13x_1^2)(x_0^2-221x_1^2)(y_0^2-53y_1^2)(y_0^3-53y_1^3)=0\\ 
			w_1^2w_2=w_0^3-16w_2^3
		\end{cases}
	\end{equation*}
	in $\PP^2\times \PP^1\times \PP^1$ with tri-homogeneous coordinates $(w_0:w_1:w_2)\times(x_0:x_1)\times(y_0:y_1).$
	For this surface $X,$ we have the following proposition.

	\begin{prop}\label{example2: main result: non-invariance of the Hasse principle with BMO for odd degree}
		For $K=\QQ$ and $L=\QQ(i),$  the smooth, projective, and geometrically connected surface $X$ has the following properties.
		\begin{itemize}
			\item The surface $X$ is a counterexample to the Hasse principle, and its failure of the Hasse principle is explained by the Brauer-Manin obstruction.
			\item The surface $X_L$ is a counterexample to the Hasse principle, but its failure of the Hasse principle cannot be explained by the Brauer-Manin obstruction.
		\end{itemize}
	\end{prop}

	\begin{proof}
		This is the same as in the proof of Theorem \ref{theorem main result: non-invariance of the Hasse principle with BMO for odd degree}.
	\end{proof}

	\section{Appendix}\label{appendix branch locus}

	\subsection{The locus $R$ in Example \ref{subsection main example2}}

	Let $f(x_0,x_1;y_0,y_1)=(x_0^2-17x_1^2)(x_0^2-13x_1^2)(x_0^2-221x_1^2)(y_0^2-53y_1^2)(y_0^3-53y_1^3)$ and  $g(x_0,x_1;y_0,y_1)=(x_0^2-41x_1^2)(x_0^2-3x_1^2)(x_0^2-123x_1^2)(y_0^2-13y_1^2)(y_0^3-41y_1^3)$ be two bi-homogeneous polynomials. 
	Let $X'$ be the locus defined by $u_0g(x_0,x_1;y_0,y_1)+u_1f(x_0,x_1;y_0,y_1)=0$ in $\PP^1\times\PP^1\times\PP^1$ with tri-homogeneous coordinates $(u_0:u_1)\times(x_0:x_1)\times(y_0:y_1).$ Let $R$ be the locus over which the composition $ X'\hookrightarrow \PP^1\times\PP^1\times\PP^1  \stackrel{pr_1}\to\PP^1$ is not smooth. We will calculate this finite locus $R.$ For $Z^f$ and $Z^g$ are curves with singularity, we have $\{(0:1),(1:0)\}\subset R.$ Next, let $u_1=1.$ We consider affine pieces of $X'.$ 
	
	Let $x_1=1$ and $y_1=1.$ Then this gives an affine piece of $X'$ by $u_0g(x_0,1;y_0,1)+f(x_0,1;y_0,1)=0$ in  $\AA^3$ with affine coordinates $(u_0,x_0,y_0).$
	For fixed $u_0,$ we use Jacobian criterion to calculate the singularity. Then $u_0$ satisfies the following equations:
	\begin{equation*}
		\begin{cases}
			u_0g(x_0,1;y_0,1)+f(x_0,1;y_0,1)=0\\
			u_0\frac{\partial g(x_0,1;y_0,1)}{\partial x_0}+\frac{\partial f(x_0,1;y_0,1)}{\partial x_0}=0\\
			u_0\frac{\partial g(x_0,1;y_0,1)}{\partial y_0}+\frac{\partial f(x_0,1;y_0,1)}{\partial y_0}=0.
		\end{cases}
	\end{equation*}
	Using computer to calculate, we have $u_0=0,$ or $-10553413/620289$ or satisfies one of the following three equations:
	
	$u_0^4 + \frac{442306822591}{11644065108}u_0^3 + \frac{15378563320976329}{38789291891025}u_0^2 + \frac{8833702498605138892}{6891564192638775}u_0 + \frac{1151555233848533056}{7244977740979225}=0,
	$
	
	$u_0^6 - \frac{795599865190}{1146914361}u_0^5 - \frac{852352831544631911}{52055002102707}u_0^4 + \frac{304535075034759072450076}{2362620380435562609 }u_0^3+ \frac{23484429357868605046160829719}{3971564859512180745729}u_0^2 + \frac{8311232379540782587276725670120990}{180257414278679347506402123}u_0 + \frac{959341731692466689320791603186246739997}{8181343261866419545273073156601}=0,
	$
	
	$
	u_0^{24} - \frac{1282484299432205}{828072168642}u_0^{23} + \frac{3122323546639431087642188987593}{5017342803508279669201200}u_0^{22}-\\ \frac{9220867294873355192932709492986698418282151}{152002022053223167005295465603491600}u_0^{21}-\\ \frac{30999681746654846295693028728045879521729132080169271161}{7580165253814008879739256663670726076436640000}u_0^{20} -\\
	\frac{45212516638352229837933187085366694204283058079529344651463540951}{2235578160515817023818667222976763042314131108360640000}u_0^{19} +\\ \frac{18075149338451367526195790572251308674104245881489906934937358864775825826797}{7325858627130126160176176715795586051349128034014867346496000000}u_0^{18} +\\ \frac{1929728458747328554854199670272434548177432513569626746401857766397194600755599}{37545025464041896570902905668452378513164281174326195150792000000}u_0^{17} +\\ \frac{10813082002346392222114449555829223571485436674784220052152543359916740809425000843}{57725476650964415977763217465245531963990082305526525044342700000000}u_0^{16} -\\ \frac{4276548928854862536400602684047575693721206178955942137599822672373098084587625072121}{887529203508577895658109468528150053946347515447470322556769012500000}u_0^{15} -\\ \frac{108138440749666040998151800754157496874442091422159386570663670108546710511792028190429}{1849019173976203949287728059433645945721557323848896505326602109375000}u_0^{14} -\\ \frac{212274800596274205751056409361280330744666951660783687161854079450560419076576526369}{2150022295321167382892707045853076681071578283545228494565816406250}u_0^{13} +\\ \frac{40608008582318322879285505067991388627920915662962695473224278401472213071209607698108369}{17334554756026912024572450557190430741139599911083404737436894775390625} u_0^{12}+\\ \frac{20176896364376034775914854511315952401902515577025172947699198733258383180655210587584504}{1155636983735127468304830037146028716075973327405560315829126318359375}u_0^{11} +\\ \frac{158963792583731661630620955844842160301301960511243192646826835314389326180471483775057248}{5778184918675637341524150185730143580379866637027801579145631591796875} u_0^{10}-\\ \frac{1270266243361503789508103099955850762203604422488301533325846718541482312691517964711577728}{5778184918675637341524150185730143580379866637027801579145631591796875}u_0^9 -\\ \frac{2666552467620466751632153917355955257796687326989260716214069955543289610262819714744442624}{1926061639558545780508050061910047860126622212342600526381877197265625}u_0^8 -\\ \frac{6882635355470258602823490665258239441168362415110817180409141527410617503796536388374673408}{1926061639558545780508050061910047860126622212342600526381877197265625}u_0^7 -\\ \frac{9084247577733305667444515416361134105121434380512329964462666543221760153655964901019889664}{1926061639558545780508050061910047860126622212342600526381877197265625}u_0^6 -\\ \frac{200506323738234616331085970009338835768870364737332320830237673581818041659073075288342528}{71335616279946140018816668959631402226911933790466686162291748046875}u_0^5 -\\ \frac{7334044106882599637223250735958076270786299935006967560560521012299445936011083944820736}{23778538759982046672938889653210467408970644596822228720763916015625}u_0^4 +\\ \frac{2816647995777364092376808098177039066661618029531562491562800915395753104512565060304896}{71335616279946140018816668959631402226911933790466686162291748046875}u_0^3 +\\ \frac{7665757353406683133913491047865070214413497147217395178477629570300922332642644328448}{880686620740075802701440357526313607739653503586008471139404296875}u_0^2 +\\ \frac{251119825007641874397975890381670516864055856553441761611195723154227892347520155648}{528411972444045481620864214515788164643792102151605082683642578125}u_0 +\\ \frac{23272944755213194420743946309558908540171345437132830639580649605274861417105719296}{2642059862220227408104321072578940823218960510758025413418212890625}=0.\\
	$
	
	Let $x_1=1$ and $y_0=1.$ Then this gives an affine piece of $X'$ by $u_0g(x_0,1;1,y_1)+f(x_0,1;1,y_1)=0$ in  $\AA^3$ with affine coordinates $(u_0,x_0,y_1).$
	For fixed $u_0,$ we use Jacobian criterion to calculate the singularity. Then $u_0$ satisfies the following equations:
	\begin{equation*}
		\begin{cases}
			u_0g(x_0,1;1,y_1)+f(x_0,1;1,y_1)=0\\
			u_0\frac{\partial g(x_0,1;1,y_1)}{\partial x_0}+\frac{\partial f(x_0,1;1,y_1)}{\partial x_0}=0\\
			u_0\frac{\partial g(x_0,1;1,y_1)}{\partial y_1}+\frac{\partial f(x_0,1;1,y_1)}{\partial y_1}=0.
		\end{cases}
	\end{equation*}
	Using computer to calculate, we have $u_0=0,$ or $-48841/15129$ or satisfies one of the following three equations:
	
	$
	u_0^4 + \frac{157460599}{21846276}u_0^3 + \frac{1949002009}{136539225}u_0^2 + \frac{398554348}{45513075}u_0^ + \frac{3125824}{15171025}=0,
	$
	
	$
	u_0^6 - \frac{795599865190}{1146914361}u_0^5 - \frac{852352831544631911}{52055002102707}u_0^4 + \frac{304535075034759072450076}{2362620380435562609}u_0^3+ \frac{23484429357868605046160829719}{3971564859512180745729}u_0^2 + \frac{8311232379540782587276725670120990}{180257414278679347506402123}u_0^ + \frac{959341731692466689320791603186246739997}{8181343261866419545273073156601}=0,
	$

	$
	u_0^{24} - \frac{1282484299432205}{828072168642}u_0^{23} + \frac{3122323546639431087642188987593}{5017342803508279669201200}u_0^{22}-\\ \frac{9220867294873355192932709492986698418282151}{152002022053223167005295465603491600}u_0^{21}-\\ \frac{30999681746654846295693028728045879521729132080169271161}{7580165253814008879739256663670726076436640000}u_0^{20} -\\ \frac{45212516638352229837933187085366694204283058079529344651463540951}{2235578160515817023818667222976763042314131108360640000}u_0^{19} +\\ \frac{18075149338451367526195790572251308674104245881489906934937358864775825826797}{7325858627130126160176176715795586051349128034014867346496000000}u_0^{18} +\\ \frac{1929728458747328554854199670272434548177432513569626746401857766397194600755599}{37545025464041896570902905668452378513164281174326195150792000000}u_0^{17} +\\ \frac{10813082002346392222114449555829223571485436674784220052152543359916740809425000843}{57725476650964415977763217465245531963990082305526525044342700000000}u_0^{16} -\\ \frac{4276548928854862536400602684047575693721206178955942137599822672373098084587625072121}{887529203508577895658109468528150053946347515447470322556769012500000}u_0^{15} -\\ \frac{108138440749666040998151800754157496874442091422159386570663670108546710511792028190429}{1849019173976203949287728059433645945721557323848896505326602109375000}u_0^{14} -\\ \frac{212274800596274205751056409361280330744666951660783687161854079450560419076576526369}{2150022295321167382892707045853076681071578283545228494565816406250}u_0^{13} +\\ \frac{40608008582318322879285505067991388627920915662962695473224278401472213071209607698108369}{17334554756026912024572450557190430741139599911083404737436894775390625} u_0^{12}+\\ \frac{20176896364376034775914854511315952401902515577025172947699198733258383180655210587584504}{1155636983735127468304830037146028716075973327405560315829126318359375}u_0^{11} +\\ \frac{158963792583731661630620955844842160301301960511243192646826835314389326180471483775057248}{5778184918675637341524150185730143580379866637027801579145631591796875} u_0^{10}-\\ \frac{1270266243361503789508103099955850762203604422488301533325846718541482312691517964711577728}{5778184918675637341524150185730143580379866637027801579145631591796875}u_0^9 -\\ \frac{2666552467620466751632153917355955257796687326989260716214069955543289610262819714744442624}{1926061639558545780508050061910047860126622212342600526381877197265625}u_0^8 -\\ \frac{6882635355470258602823490665258239441168362415110817180409141527410617503796536388374673408}{1926061639558545780508050061910047860126622212342600526381877197265625}u_0^7 -\\ \frac{9084247577733305667444515416361134105121434380512329964462666543221760153655964901019889664}{1926061639558545780508050061910047860126622212342600526381877197265625}u_0^6 -\\ \frac{200506323738234616331085970009338835768870364737332320830237673581818041659073075288342528}{71335616279946140018816668959631402226911933790466686162291748046875}u_0^5 -\\ \frac{7334044106882599637223250735958076270786299935006967560560521012299445936011083944820736}{23778538759982046672938889653210467408970644596822228720763916015625}u_0^4 +\\ \frac{2816647995777364092376808098177039066661618029531562491562800915395753104512565060304896}{71335616279946140018816668959631402226911933790466686162291748046875}u_0^3 +\\ \frac{7665757353406683133913491047865070214413497147217395178477629570300922332642644328448}{880686620740075802701440357526313607739653503586008471139404296875}u_0^2 +\\ \frac{251119825007641874397975890381670516864055856553441761611195723154227892347520155648}{528411972444045481620864214515788164643792102151605082683642578125}u_0 +\\ \frac{23272944755213194420743946309558908540171345437132830639580649605274861417105719296}{2642059862220227408104321072578940823218960510758025413418212890625}=0.\\
	$
	
	Let $x_0=1$ and $y_1=1.$ Then this gives an affine piece of $X'$ by $u_0g(1,x_1;y_0,1)+f(1,x_1;y_0,1)=0$ in  $\AA^3$ with affine  coordinates $(u_0,x_1,y_0).$
	For fixed $u_0,$ we use Jacobian criterion to calculate the singularity. Then $u_0$ satisfies the following equations:
	\begin{equation*}
		\begin{cases}
			u_0g(1,x_1;y_0,1)+f(1,x_1;y_0,1)=0\\
			u_0\frac{\partial g(1,x_1;y_0,1)}{\partial x_1}+\frac{\partial f(1,x_1;y_0,1)}{\partial x_1}=0\\
			u_0\frac{\partial g(1,x_1;y_0,1)}{\partial y_0}+\frac{\partial f(1,x_1;y_0,1)}{\partial y_0}=0.
		\end{cases}
	\end{equation*}
	Using computer to calculate, we have $u_0=0,$ or $-2809/533$ or satisfies one of the following three equations:

	$
	u_0^4 + \frac{442306822591}{11644065108}u_0^3 + \frac{15378563320976329}{38789291891025}u_0^2 + \frac{8833702498605138892}{6891564192638775}u_0^ + \frac{1151555233848533056}{7244977740979225}=0,
	$
	
	$
	u_0^6 - \frac{16289590}{75809}u_0^5 - \frac{357314231}{227427}u_0^4 + \frac{2613868156}{682281}u_0^3 + \frac{4127069879}{75809}u_0^2 + \frac{29904922990}{227427}u_0^ +\frac{70675038317}{682281}=0,
	$

	$
	u_0^{24} - \frac{1282484299432205}{828072168642}u_0^{23} + \frac{3122323546639431087642188987593}{5017342803508279669201200}u_0^{22}-\\ \frac{9220867294873355192932709492986698418282151}{152002022053223167005295465603491600}u_0^{21}-\\ \frac{30999681746654846295693028728045879521729132080169271161}{7580165253814008879739256663670726076436640000}u_0^{20} -\\ \frac{45212516638352229837933187085366694204283058079529344651463540951}{2235578160515817023818667222976763042314131108360640000}u_0^{19} +\\ \frac{18075149338451367526195790572251308674104245881489906934937358864775825826797}{7325858627130126160176176715795586051349128034014867346496000000}u_0^{18} +\\ \frac{1929728458747328554854199670272434548177432513569626746401857766397194600755599}{37545025464041896570902905668452378513164281174326195150792000000}u_0^{17} +\\ \frac{10813082002346392222114449555829223571485436674784220052152543359916740809425000843}{57725476650964415977763217465245531963990082305526525044342700000000}u_0^{16} -\\ \frac{4276548928854862536400602684047575693721206178955942137599822672373098084587625072121}{887529203508577895658109468528150053946347515447470322556769012500000}u_0^{15} -\\ \frac{108138440749666040998151800754157496874442091422159386570663670108546710511792028190429}{1849019173976203949287728059433645945721557323848896505326602109375000}u_0^{14} -\\ \frac{212274800596274205751056409361280330744666951660783687161854079450560419076576526369}{2150022295321167382892707045853076681071578283545228494565816406250}u_0^{13} +\\ \frac{40608008582318322879285505067991388627920915662962695473224278401472213071209607698108369}{17334554756026912024572450557190430741139599911083404737436894775390625} u_0^{12}+\\ \frac{20176896364376034775914854511315952401902515577025172947699198733258383180655210587584504}{1155636983735127468304830037146028716075973327405560315829126318359375}u_0^{11} +\\ \frac{158963792583731661630620955844842160301301960511243192646826835314389326180471483775057248}{5778184918675637341524150185730143580379866637027801579145631591796875} u_0^{10}-\\ \frac{1270266243361503789508103099955850762203604422488301533325846718541482312691517964711577728}{5778184918675637341524150185730143580379866637027801579145631591796875}u_0^9 -\\ \frac{2666552467620466751632153917355955257796687326989260716214069955543289610262819714744442624}{1926061639558545780508050061910047860126622212342600526381877197265625}u_0^8 -\\ \frac{6882635355470258602823490665258239441168362415110817180409141527410617503796536388374673408}{1926061639558545780508050061910047860126622212342600526381877197265625}u_0^7 -\\ \frac{9084247577733305667444515416361134105121434380512329964462666543221760153655964901019889664}{1926061639558545780508050061910047860126622212342600526381877197265625}u_0^6 -\\ \frac{200506323738234616331085970009338835768870364737332320830237673581818041659073075288342528}{71335616279946140018816668959631402226911933790466686162291748046875}u_0^5 -\\ \frac{7334044106882599637223250735958076270786299935006967560560521012299445936011083944820736}{23778538759982046672938889653210467408970644596822228720763916015625}u_0^4 +\\ \frac{2816647995777364092376808098177039066661618029531562491562800915395753104512565060304896}{71335616279946140018816668959631402226911933790466686162291748046875}u_0^3 +\\ \frac{7665757353406683133913491047865070214413497147217395178477629570300922332642644328448}{880686620740075802701440357526313607739653503586008471139404296875}u_0^2 +\\ \frac{251119825007641874397975890381670516864055856553441761611195723154227892347520155648}{528411972444045481620864214515788164643792102151605082683642578125}u_0 +\\ \frac{23272944755213194420743946309558908540171345437132830639580649605274861417105719296}{2642059862220227408104321072578940823218960510758025413418212890625}=0.\\
	$

	Let $x_0=1$ and $y_0=1.$ Then this gives an affine piece of $X'$ by $u_0g(1,x_1;1,y_1)+f(1,x_1;1,y_1)=0$ in  $\AA^3$ with affine  coordinates $(u_0,x_1,y_1).$
	For fixed $u_0,$ we use Jacobian criterion to calculate the singularity. Then $u_0$ satisfies the following equations:
	\begin{equation*}
		\begin{cases}
			u_0g(1,x_1;1,y_1)+f(1,x_1;1,y_1)=0\\
			u_0\frac{\partial g(1,x_1;1,y_1)}{\partial x_1}+\frac{\partial f(1,x_1;1,y_1)}{\partial x_1}=0\\
			u_0\frac{\partial g(1,x_1;1,y_1)}{\partial y_1}+\frac{\partial f(1,x_1;1,y_1)}{\partial y_1}=0.
		\end{cases}
	\end{equation*}
	Using computer to calculate, we have $u_0=0,$ or $-1$ or satisfies one of the following three equations:
	
	$u_0^4 + \frac{157460599}{21846276}u_0^3 + \frac{1949002009}{136539225}u_0^2 + \frac{398554348}{45513075}u_0^ + \frac{3125824}{15171025}=0,
	$
	
	$u_0^6 - \frac{16289590}{75809}u_0^5 - \frac{357314231}{227427}u_0^4 + \frac{2613868156}{682281}u_0^3 + \frac{4127069879}{75809}u_0^2 + \frac{29904922990}{227427}u_0^ + \frac{70675038317}{682281}=0,
	$

	$
	u_0^{24} - \frac{1282484299432205}{828072168642}u_0^{23} + \frac{3122323546639431087642188987593}{5017342803508279669201200}u_0^{22}-\\ \frac{9220867294873355192932709492986698418282151}{152002022053223167005295465603491600}u_0^{21}-\\ \frac{30999681746654846295693028728045879521729132080169271161}{7580165253814008879739256663670726076436640000}u_0^{20} -\\ \frac{45212516638352229837933187085366694204283058079529344651463540951}{2235578160515817023818667222976763042314131108360640000}u_0^{19} +\\ \frac{18075149338451367526195790572251308674104245881489906934937358864775825826797}{7325858627130126160176176715795586051349128034014867346496000000}u_0^{18} +\\ \frac{1929728458747328554854199670272434548177432513569626746401857766397194600755599}{37545025464041896570902905668452378513164281174326195150792000000}u_0^{17} +\\ \frac{10813082002346392222114449555829223571485436674784220052152543359916740809425000843}{57725476650964415977763217465245531963990082305526525044342700000000}u_0^{16} -\\ \frac{4276548928854862536400602684047575693721206178955942137599822672373098084587625072121}{887529203508577895658109468528150053946347515447470322556769012500000}u_0^{15} -\\ \frac{108138440749666040998151800754157496874442091422159386570663670108546710511792028190429}{1849019173976203949287728059433645945721557323848896505326602109375000}u_0^{14} -\\ \frac{212274800596274205751056409361280330744666951660783687161854079450560419076576526369}{2150022295321167382892707045853076681071578283545228494565816406250}u_0^{13} +\\ \frac{40608008582318322879285505067991388627920915662962695473224278401472213071209607698108369}{17334554756026912024572450557190430741139599911083404737436894775390625} u_0^{12}+\\ \frac{20176896364376034775914854511315952401902515577025172947699198733258383180655210587584504}{1155636983735127468304830037146028716075973327405560315829126318359375}u_0^{11} +\\ \frac{158963792583731661630620955844842160301301960511243192646826835314389326180471483775057248}{5778184918675637341524150185730143580379866637027801579145631591796875} u_0^{10}-\\ \frac{1270266243361503789508103099955850762203604422488301533325846718541482312691517964711577728}{5778184918675637341524150185730143580379866637027801579145631591796875}u_0^9 -\\ \frac{2666552467620466751632153917355955257796687326989260716214069955543289610262819714744442624}{1926061639558545780508050061910047860126622212342600526381877197265625}u_0^8 -\\ \frac{6882635355470258602823490665258239441168362415110817180409141527410617503796536388374673408}{1926061639558545780508050061910047860126622212342600526381877197265625}u_0^7 -\\ \frac{9084247577733305667444515416361134105121434380512329964462666543221760153655964901019889664}{1926061639558545780508050061910047860126622212342600526381877197265625}u_0^6 -\\ \frac{200506323738234616331085970009338835768870364737332320830237673581818041659073075288342528}{71335616279946140018816668959631402226911933790466686162291748046875}u_0^5 -\\ \frac{7334044106882599637223250735958076270786299935006967560560521012299445936011083944820736}{23778538759982046672938889653210467408970644596822228720763916015625}u_0^4 +\\ \frac{2816647995777364092376808098177039066661618029531562491562800915395753104512565060304896}{71335616279946140018816668959631402226911933790466686162291748046875}u_0^3 +\\ \frac{7665757353406683133913491047865070214413497147217395178477629570300922332642644328448}{880686620740075802701440357526313607739653503586008471139404296875}u_0^2 +\\ \frac{251119825007641874397975890381670516864055856553441761611195723154227892347520155648}{528411972444045481620864214515788164643792102151605082683642578125}u_0 +\\ \frac{23272944755213194420743946309558908540171345437132830639580649605274861417105719296}{2642059862220227408104321072578940823218960510758025413418212890625}=0.\\
	$
	
	In summary, the locus $R=\{(0:1),(1:0),(-10553413:620289),(-48841:15129), (-2809:533),(-1:1)\}\cup \{(u_0:1)|u_0$ satisfies one of the following five equations $\}.$

	$u_0^4 + \frac{442306822591}{11644065108}u_0^3 + \frac{15378563320976329}{38789291891025}u_0^2 + \frac{8833702498605138892}{6891564192638775}u_0 + \frac{1151555233848533056}{7244977740979225}=0,
	$
	
	$
	u_0^4 + \frac{157460599}{21846276}u_0^3 + \frac{1949002009}{136539225}u_0^2 + \frac{398554348}{45513075}u_0^ + \frac{3125824}{15171025}=0,
	$
	
	$u_0^6 - \frac{795599865190}{1146914361}u_0^5 - \frac{852352831544631911}{52055002102707}u_0^4 + \frac{304535075034759072450076}{2362620380435562609 }u_0^3+ \frac{23484429357868605046160829719}{3971564859512180745729}u_0^2 + \frac{8311232379540782587276725670120990}{180257414278679347506402123}u_0 + \frac{959341731692466689320791603186246739997}{8181343261866419545273073156601}=0,
	$

	$
	u_0^6 - \frac{16289590}{75809}u_0^5 - \frac{357314231}{227427}u_0^4 + \frac{2613868156}{682281}u_0^3 + \frac{4127069879}{75809}u_0^2 + \frac{29904922990}{227427}u_0^ +\frac{70675038317}{682281}=0,
	$

	$
	u_0^{24} - \frac{1282484299432205}{828072168642}u_0^{23} + \frac{3122323546639431087642188987593}{5017342803508279669201200}u_0^{22}-\\ \frac{9220867294873355192932709492986698418282151}{152002022053223167005295465603491600}u_0^{21}-\\ \frac{30999681746654846295693028728045879521729132080169271161}{7580165253814008879739256663670726076436640000}u_0^{20} -\\ \frac{45212516638352229837933187085366694204283058079529344651463540951}{2235578160515817023818667222976763042314131108360640000}u_0^{19} +\\ \frac{18075149338451367526195790572251308674104245881489906934937358864775825826797}{7325858627130126160176176715795586051349128034014867346496000000}u_0^{18} +\\ \frac{1929728458747328554854199670272434548177432513569626746401857766397194600755599}{37545025464041896570902905668452378513164281174326195150792000000}u_0^{17} +\\ \frac{10813082002346392222114449555829223571485436674784220052152543359916740809425000843}{57725476650964415977763217465245531963990082305526525044342700000000}u_0^{16} -\\ \frac{4276548928854862536400602684047575693721206178955942137599822672373098084587625072121}{887529203508577895658109468528150053946347515447470322556769012500000}u_0^{15} -\\ \frac{108138440749666040998151800754157496874442091422159386570663670108546710511792028190429}{1849019173976203949287728059433645945721557323848896505326602109375000}u_0^{14} -\\ \frac{212274800596274205751056409361280330744666951660783687161854079450560419076576526369}{2150022295321167382892707045853076681071578283545228494565816406250}u_0^{13} +\\ \frac{40608008582318322879285505067991388627920915662962695473224278401472213071209607698108369}{17334554756026912024572450557190430741139599911083404737436894775390625} u_0^{12}+\\ \frac{20176896364376034775914854511315952401902515577025172947699198733258383180655210587584504}{1155636983735127468304830037146028716075973327405560315829126318359375}u_0^{11} +\\ \frac{158963792583731661630620955844842160301301960511243192646826835314389326180471483775057248}{5778184918675637341524150185730143580379866637027801579145631591796875} u_0^{10}-\\ \frac{1270266243361503789508103099955850762203604422488301533325846718541482312691517964711577728}{5778184918675637341524150185730143580379866637027801579145631591796875}u_0^9 -\\ \frac{2666552467620466751632153917355955257796687326989260716214069955543289610262819714744442624}{1926061639558545780508050061910047860126622212342600526381877197265625}u_0^8 -\\ \frac{6882635355470258602823490665258239441168362415110817180409141527410617503796536388374673408}{1926061639558545780508050061910047860126622212342600526381877197265625}u_0^7 -\\ \frac{9084247577733305667444515416361134105121434380512329964462666543221760153655964901019889664}{1926061639558545780508050061910047860126622212342600526381877197265625}u_0^6 -\\ \frac{200506323738234616331085970009338835768870364737332320830237673581818041659073075288342528}{71335616279946140018816668959631402226911933790466686162291748046875}u_0^5 -\\ \frac{7334044106882599637223250735958076270786299935006967560560521012299445936011083944820736}{23778538759982046672938889653210467408970644596822228720763916015625}u_0^4 +\\ \frac{2816647995777364092376808098177039066661618029531562491562800915395753104512565060304896}{71335616279946140018816668959631402226911933790466686162291748046875}u_0^3 +\\ \frac{7665757353406683133913491047865070214413497147217395178477629570300922332642644328448}{880686620740075802701440357526313607739653503586008471139404296875}u_0^2 +\\ \frac{251119825007641874397975890381670516864055856553441761611195723154227892347520155648}{528411972444045481620864214515788164643792102151605082683642578125}u_0 +\\ \frac{23272944755213194420743946309558908540171345437132830639580649605274861417105719296}{2642059862220227408104321072578940823218960510758025413418212890625}=0.\\
	$
	
	Let $(u_0:1)$ be a closed point in $R,$ then the degree $[\QQ(u_0):\QQ]\in \{1,4,6,24\}.$
	% consists of 12-points. So the points of this branch locus conjugate to each other over $\QQ.$
	%%%%%%%%%%%%%%%%%%%%%%%%%%%%%

	\begin{footnotesize}
		\noindent\textbf{Acknowledgements.} The author would like to thank my thesis advisor Y. Liang for proposing the related problems, papers and many fruitful discussions. This paper was inspired by the work of Harpaz and Skorobogatov \cite{HS14}. The author was partially supported by NSFC Grant No. 12071448.
	\end{footnotesize}


	%	 \bibliographystyle{alpha}
	%\bibliography{../../../mybib}

\end{document}